\documentclass[11pt,a4paper]{amsart}
\usepackage{amssymb,amsmath}
\title[The Joint Spectral Flow and Localization of the Indices]{The Joint Spectral Flow and Localization of the Indices of Elliptic Operators}
\author{Yosuke KUBOTA}
\date{}
\keywords{Index theory, spectral flow, localization, connective $K$-theory, $KK$-theory.}
\subjclass[2010]{Primary 19K56; Secondary 19K35, 19L41.}

\usepackage{amsmath,amssymb}
\usepackage{mathrsfs}
\usepackage[dvipdfmx]{graphicx,color}
\usepackage{wrapfig}
\usepackage{calc}
\usepackage{extarrows}
\usepackage{amsthm}
\usepackage{latexsym}
\usepackage{slashed}
\usepackage{hyperref}

  \makeatletter
  \@addtoreset{equation}{section}
  \makeatother

\usepackage{framed}
\usepackage{ascmac}
\usepackage[all]{xy}

\theoremstyle{definition}
\newtheorem{defn}[equation]{Definition}
\theoremstyle{plain}
\newtheorem{thm}[equation]{Theorem}
\newtheorem{prp}[equation]{Proposition}
\newtheorem{lem}[equation]{Lemma}
\newtheorem{cor}[equation]{Corollary}

\theoremstyle{remark}
\newtheorem{remk}[equation]{Remark}
\newtheorem{exmp}[equation]{Example}
\newtheorem*{rmk*}{Remark}

\DeclareMathAlphabet{\mathfrak}{U}{euf}{m}{n}
\SetMathAlphabet{\mathfrak}{bold}{U}{euf}{b}{n}
\newcommand{\mf}[1]{\mathfrak{#1}}
\newcommand{\ms}[1]{\mathscr{#1}}

\newcommand{\mbb}[1]{\mathbb{#1}}
\newcommand{\mc}[1]{\mathcal{#1}}

\newcommand{\ma}[1]{\begin{align*} #1 \end{align*}}
\newcommand{\maa}[1]{\begin{align} #1 \end{align}}

\newcommand{\real}{\mathbb R}
\newcommand{\comp}{\mathbb C}
\newcommand{\zahl}{\mathbb Z}
\newcommand{\quot}{\mathbb Q}

\DeclareMathOperator{\Hom}{Hom}

\DeclareMathOperator{\End}{End}

\newcommand{\id}{{\rm id}}

\DeclareMathOperator{\dom}{dom}

\DeclareMathOperator{\ind}{ind}

\newcommand{\bk} [1]{\left(  #1 \right)}
\newcommand{\ebk}[1]{\left<  #1 \right>}
\newcommand{\mbk}[1]{\left\{ #1 \right\}}
\newcommand{\lbk}[1]{\left[  #1 \right]}

\newcommand{\ssbk}[1]{\left\| #1 \right\|}

\newcommand{\pmx}[1]{\begin{pmatrix} #1 \end{pmatrix}}%

\newcommand{\ch}{{\rm ch}}

\newcommand{\ral}{\longrightarrow}

\newcommand{\xra}{\xrightarrow}

\newcommand{\Td}{{\rm Td}}

\newcommand{\Map}{{\rm Map}}

\newcommand{\Cliff}{\mathbb{C}\ell }
\newcommand{\pt}{{\rm pt}}

\DeclareMathOperator{\jsf}{jsf}
\DeclareMathOperator{\Sym}{Sym}
\renewcommand{\Re}{{\rm Re} \hspace{0.2em}}
\renewcommand{\Im}{{\rm Im} \hspace{0.2em}}

\begin{document}

\maketitle
\begin{abstract}
We introduce the notion of the joint spectral flow, which is a generalization of the spectral flow, by using Segal's model of the connective $K$-theory spectrum. We apply it for some localization results of indices motivated by Witten's deformation of Dirac operators and rephrase some analytic techniques in terms of topology. 
\end{abstract}

\tableofcontents

\section{Introduction}
In this paper we give a topological viewpoint for the index and its localization phenomena of elliptic operators on certain fiber bundles using the notion of the joint spectral flow, which is a generalization of that of spectral flow introduced by Atiyah-Patodi-Singer~\cite{AtiyahPatodiSinger1976}. It has various generalizations, for example, higher spectral flow given by Dai-Zhang~\cite{DaiZhang1998}, and noncommutative spectral flow by Leichtnam-Piazza~\cite{LeichtnamPiazza2003} and Wahl~\cite{Wahl2007}. However, what we introduce here is a completely different, new generalization. 

The spectral flow for a one-parameter family of self-adjoint operators is an integer counting the number of eigenvalues crossing over zero with multiplicity. In geometric situations, it is related to the index of some Fredholm operator as shown by Atiyah-Patodi-Singer~\cite{AtiyahPatodiSinger1976} as follows. For a one parameter family of self-adjoint elliptic differential operators $D_{t}$ of first order ($t\in S^{1}$) on $\Gamma (Y,E)$, where $Y$ is a closed manifold and $E$ is a hermitian vector bundle on $Y$, a first order differential operator $d/dt+D_{t}$ on $\Gamma (Y\times S^{1}, \pi^{*}E)$ is also elliptic and its index coincides with the spectral flow. Its proof is given essentially by the family's index theorem on the closed $1$-dimensional manifold $S^1$.

The joint spectral flow deals with an $n$-parameter family of $n$-tuples of mutually commuting self-adjoint operators and their joint spectra. We deal with continuous or smooth families of commuting Fredholm $n$-tuples, which are defined in Definition \ref{def:tuple}, and the ``Dirac operators'' associated with them. In the special case of $n=1$, it coincides with the usual spectral flow. We also relate it with the index of some elliptic operator as is the case of the ordinal spectral flow. 
\theoremstyle{plain}
\newtheorem*{fake1}{Theorem \ref{thm:jsftwisted}}
\begin{fake1}
Let $B$ be a closed $n$-dimensional $Spin^c$ manifold, $Z \to M \to B$ a smooth fiber bundle over $B$ such that the total space $M$ is also a $Spin^c$-manifold, $E$ a smooth complex vector bundle over $M$, $V$ an $n$-dimensional $Spin^c$ vector bundle over $B$. For a bundle map $\mbk{D_v(x)}$ from $V \setminus \mbk{0}$ to the bundle of fiberwise pseudodifferential operators $\Psi _f ^1 (M,E)$ satisfying Condition \ref{cond:commtwist}, the following formula holds.
\ma{\ind (\pi ^* \slashed{\mf{D}}_B + D(x))=\jsf (\mbk{D(x)}).}
\end{fake1}
The proof also works in a similar way to the original one. The crucial theorem introduced by Segal~\cite{Segal1977} is that the space of $n$-tuples of mutually commuting compact self-adjoint operators is a model for the spectrum of the connective $K$-group. 

The joint spectral flow and its index formula implies some localization results. In \cite{Witten1982} E. Witten reinterpreted and reproved some localization formulas for the indices of Dirac operators from the viewpoint of supersymmetry. He deformed Dirac operators by adding potential terms coming from Morse functions or Killing vectors. Recently Fujita-Furuta-Yoshida~\cite{FujitaFurutaYoshida2010} used its infinite dimensional analogue to localize the Riemann-Roch numbers of certain completely integrable systems and their prequantum data on their Bohr-Sommerfeld fibers. In this case the indices of Dirac operators on fiber bundles localize on some special fibers instead of points. Here we relate them with our joint spectral flow and give a topological viewpoint for this analytic way of localization. A strong point of our method is that we give a precise way to compute the multiplicity at each point on which the index localizes. As a consequence we reprove and generalize theorems of Witten and Fujita-Furuta-Yoshida.

\newtheorem*{fake2}{Corollary \ref{cor:FFY}}
\begin{fake2}[Andersen~\cite{Andersen1997}, Fujita-Furuta-Yoshida~\cite{FujitaFurutaYoshida2010}]
Let $(X,\omega)$ be a symplectic manifold of dimension $2n$, $\mathbb{T}^n \to X \to B$ a Lagrangian fiber bundle, and $(L,\nabla ^L,h)$ its prequantum data. Then its Riemann-Roch number $RR(M,L)$ coincides with the number of Bohr-Sommerfeld fibers.
\end{fake2}

Finally we consider an operator-theoretic problem. 

Unfortunately, there are not many examples of geometrically important operators (for example Dirac operators) represented as Dirac operators associated with commuting Fredholm $n$-tuple coming from differential operators. Compared the case that their principal symbols are ``decomposed'' as the sum of commuting $n$-tuples, which is the easiest case because it is realized when their tangent bundles are decomposed, the case that the Dirac operators themselves are decomposed is much more difficult because it requires some integrability of decompositions of tangent bundles. However, the bounded operators $\slashed{D}(1+\slashed{D}^2)^{-1/2}$ associated with the Dirac operators $\slashed{D}$ and zeroth order pseudodifferential operators are much easier to deal with than first order differential operators. We glue two commuting $n$-tuples of pseudodifferential operators by using topological methods to show that the family's indices are complete obstructions of decomposing property of families of Dirac operators. Here the theory of extensions of $C^*$-algebra and Cuntz's quasihomomorphism plays an important role.

\newtheorem*{fake3}{Theorem \ref{thm:decomp}}
\begin{fake3}
Let $Z \to M \to B$ be a fiber bundle. We assume that there are vector bundles $V_1,\ldots,V_l$ on $B$ and $E_1,\ldots,E_l$ on $M$ such that the vertical tangent bundle $T_VM$ is isomorphic to $\pi ^*V_1 \otimes E_1 \oplus \cdots \oplus \pi ^*V_l \otimes E_l$. Then its fiberwise Dirac operator $\slashed{D}_f^E$ is $n$-decomposable (in the sense of Definition \ref{def:decomp}) if and only if the family's index $\ind (\slashed{D}_f^E)$ is in the image of $K^n(B, B^{(n-1)}) \to K^n(B)$, or equivalently the image of $\tilde{k}^n(B) \to K^n(B)$. 
\end{fake3}

This paper is organized as follows. In Section \ref{section:2}, we relate Segal's description of the connective $K$-theory with the theory of Fredholm operators. In Section \ref{section:3}, we introduce the notion of the joint spectral flow and prove its index formula. In Section \ref{section:4}, we apply the theory and reprove or generalize some classical facts. In Section \ref{section:5} we deal with a decomposing problem of Dirac operators and give an index theoretic complete obstruction.

\noindent {\bf Conventions.}
We use the following notations throughout this paper. 

First, any topological space is assumed to be locally compact and Hausdorff unless otherwise noted (there are some exceptions, which are mentioned individually). 

Second, we use some terms of topology as follows.
For a based space $(X,*)$, we denote by $\Sigma X$ the suspension $X \times S^1 /(X \times *_{S^1} \cup *_X \times S^1)$ and by $\Omega X$ the reduced loop space $\Map ((S^1,*),(X,*))$. 
On the other hand, for an unbased space $X$ we denote by $\Sigma X$ (resp. $IX$) the direct sum $X \times (0,1)$ (resp. $X \times [0,1]$). Similarly, for a $C^*$-algebra $A$ we denote by $\Sigma A$ (resp. $IA$) its suspension $A \otimes C_0(0,1)$ (resp. $A \otimes C[0,1]$). In particular, we denote by only $\Sigma $ (resp. $I$) the topological space $(0,1)$ or the $C^*$-algebra $C_0(0,1)$ (resp. $[0,1]$ or $C[0,1]$).

\noindent {\bf Acknowledgement.} 
The author would like to thank his supervisor Professor Yasuyuki Kawahigashi for his support and encouragement. He also would like to thank his sub-supervisor Professor Mikio Furuta for suggesting the problem and several helpful comments. This work was supported by the Program for Leading Graduate Schools, MEXT, Japan.

\section{Fredholm picture of the connective K-theory}\label{section:2}

In this section we first summarize the notion of the connective $K$-theory and its relation to operator algebras according to \cite{Segal1977} and \cite{DadarlatNemethi1990}. Then we connect it with a model of the $K$-theory spectrum that is related to the space of Fredholm operators. Finally we generalize the theory for the twisted case.
It is fundamental in order to describe the notion of the joint spectral flow.

Let $\mbk{H^i}_{i \in \zahl}$ be a generalized cohomology theory. We say $\mbk{h_i}_{i \in \zahl}$ is the connective cohomology theory associated to $\mbk{H_i}$ if it is a generalized cohomology theory that satisfies the following two properties.
\begin{enumerate}
\item There is a canonical natural transformation $h^i \to H^i$ that induces an isomorphism $h^i(\pt) \to H^i(\pt)$ for $i \leq 0$.
\item We have $h^i(\pt)=0$ for $i > 0$.
\end{enumerate}
The (reduced) connective $K$-theory is the connective cohomology theory that is associated to the (reduced) $K$-theory.

Segal~\cite{Segal1977} gave an explicit realization of connective $K$-theory spectra by using methods of operator algebras. 

For a pair of compact Hausdorff spaces $(X,A)$, we denote by $F(X,A)$ the configuration space with labels in finite dimensional subspaces of a fixed (separable infinite dimensional) Hilbert space. More precisely, an element of $F(X,A)$ is a pair $(S,\mbk{V_x}_{x \in S})$ where $S$ is a countable subset of $X \setminus A$ whose cluster points are all in $A$ and each $V_x$ is a nonzero finite dimensional subspace of a Hilbert space $\mc{H}$ such that $V_x$ and $V_y$ are orthogonal if $x \neq y$. It is a non-locally compact topological space with its canonical topology that satisfies the following.
\begin{enumerate}
\item When two sequences $\mbk{x_i}$, $\mbk{y_i}$ converge to the same point $z$ and $V_z$ is the limit of $\mbk{V_{i,x_i} \oplus V_{i,y_i}}$, the limit of $(\mbk{x_i,y_i}, \mbk{V_{i,x_i}, V_{i,y_i}})$ is $(\mbk{z}, \mbk{V_z})$.
\item When all cluster points of a sequence $\mbk{x_i}$ are in $A$, the limit of $(\mbk{x_i} , \mbk{V_{i,x_i}})$ is $(\emptyset , \emptyset)$.
\end{enumerate}

Then the following holds for this topological space. 

\begin{prp}
Let $(X,A)$ be a pair of compact Hausdorff spaces. We assume that $X$ is connected, $A$ is path-connected, and $A$ is a neighborhood deformation retract in $X$. Then the space $F(X,A)$ is homotopy equivalent to its subspace $F_{\rm fin}(X,A):=\mbk{(S,\mbk{V_x}_{x \in S}) \in F(X,A); \# S<\infty}$ and a sequence $F_{\rm fin}(A,*) \to F_{\rm fin}(X,*) \to F_{\rm fin}(X,A)$ is a quasifibration. Here morphisms are induced by continuous maps $(A,*) \to (X,*) \to (X,A)$. Hence the map $F(X,A) \to \Omega F(SX,SA)$ induces a homotopy equivalence.
\end{prp}

\begin{proof}
See Proposition 1.3 of Segal~\cite{Segal1977} and Section 3.1 of D{\u{a}}d{\u{a}}rlat-N{\'e}methi~\cite{DadarlatNemethi1990}.
\end{proof}

This means that $\mbk{F(S^n,*)}_{n=1,2,\ldots}$ is an $\Omega$-spectrum and hence homotopy classes of continuous maps to it realize some cohomology theory. 

Now we introduce two other non-locally compact spaces. First, let $F_n(\mc{H})$ be a space of $(n+1)$-tuples $\mbk{T_i}_{i=0,...,n}$ of self-adjoint bounded operators on $\mc{H}$ that satisfy the following.
\begin{enumerate}
\item The operator $T^2:=\sum T_i^2$ is equal to the identity.
\item The operator $T_i$ commutes with $T_j$ for any $i$ and $j$.
\item The operators $T_i$ ($i=1,2,\ldots,n$) and $T_0-1$ are compact.
\end{enumerate}

Then there is a canonical one-to-one correspondence between $F_n(\mc{H})$ and $F(S^n,*)$. If we have an element $(S, \mbk{V_x})$ of $F(S^n,*)$, then we obtain a $(T_0,\ldots,T_n)$ by setting $T_i:=\sum _{x \in S} x_i P_{V_x}$ where $P_V$ is the orthogonal projection onto $V$ and $x_i$ the $i$-th coordinate of $x$ in $S^n \subset \real ^{n+1}$. Conversely if we have an element $(T_0,\ldots,T_n)$ in $F_n(\mc{H})$, then we obtain data of joint spectra and the eigenspaces because they are simultaneously diagonalizable. Actually this correspondence is homeomorphic. 

On the other hand, if we have an element $(T_0,\ldots, T_n) \in F_n (\mc{H})$, then there is a canonical inclusion from the spectrum of the abelian $C^*$-algebra $C^*(T_0,\ldots,T_n)$ into the unit sphere of $\real ^{n+1}$ according to condition 1. It gives a $*$-homomorphism $C(S^n) \to \mbb{B}(\mc{H})$ sending $x_i$ to $T_i$. Now by virtue of condition 2, the image of its restriction to $C_0(S^n \setminus \mbk{*})$ is in the compact operator algebra $\mbb{K}=\mbb{K}(\mc{H})$. Conversely, if we have a $*$-homomorphism $\varphi : C_0(\real ^n) \to \mbb{K}$, then we obtain an element $(\varphi (x_0), \varphi (x_1),\ldots, \varphi (x_n))$ in $F_n(\mc{H})$. This gives a canonical one-to-one correspondence between $F_n(\mc{H})$ and $\Hom (C_0(\real ^n),\mbb{K})$. This correspondence is also a homeomorphism if we equip $\Hom (C_0(\real ^n) , \mbb{K})$ with the strong topology. Moreover, a continuous family of $*$-homomorphisms $\mbk{\varphi _x}_{x \in X}$ parametrized by a finite CW-complex $X$ is regarded as a $*$-homomorphism $C_0(\real ^n) \to C(X) \otimes \mbb{K} \cong C(X,\mbb{K})$.

\begin{prp}[\cite{Segal1977}, \cite{DadarlatNemethi1990}]
Let $X$ be a finite CW-complex and $n \in \zahl _{>0}$. The three sets
\begin{enumerate}
\item $[X, F(S^n,*)]$
\item $[X, F_n(\mc{H})]$
\item $[C_0(\real ^n) , C(X) \otimes \mbb{K}]$
\end{enumerate}
are canonically mutually isomorphic and form the $n$-th reduced connective $K$-group $\tilde{k}^n(X)$.
Here the first two are the sets of homotopy classes of continuous maps and the third is that of homotopy classes of $*$-homomorphisms.
\end{prp}

\begin{proof}
We have already seen that these three sets are canonically isomorphic and $\mbk{F(S^n,*)}_{n=1,2,\ldots}$ is an $\Omega$-spectrum. The desired canonical natural transform is a canonical map $\Phi$ from $[C_0(\real ^n), C(X) \otimes \mbb{K}]$ to $KK(C_0(\real^n),C(X) \otimes \mbb{K}) \cong K^n(X)$ that sends a homotopy class $[\varphi ]$ to $[\mc{H}\otimes C(X) , \varphi , 0]$. Hence we only have to compute $\pi _i (F(S^n,*))$. First for a general $C^*$-algebra $A$, $[C_0(\real) , A]$ is isomorphic to $K_1(A)$ because a $*$-homomorphism from $C_0(\real)$ to $A$ is determined by a unitary operator. Hence $[X,F(S^1,*)]$ is isomorphic to $K^1(X)$. In the case $i \geq n$ we have $\pi _i (F(S^n,*)) \cong \pi _{i-n+1} (F(S^1,*)) \cong K^1(\real ^{i-n+1})$ that is $\zahl$ when $i-n$ is even and $0$ when $i-n$ is odd. In the case $i < n $ we have $\pi _i(F(S^n,*)) \cong \pi _0 (F(S^{n-i},*)) \cong 0$ because $F(S^{n-i},*)$ is connected.
\end{proof}

Next we relate this picture to a realization of $K$-theory that uses the space of Fredholm operators.

Atiyah gave a realization of the $K$-theory spectrum in \cite{AtiyahSinger1969}. Let $\Cliff _n$ be the complex Clifford algebra associated to $\comp ^n$ and its canonical inner product, $e_1,\ldots,e_n$ be its canonical self-adjoint generators with relations $e_ie_j+e_je_i =2\delta _{ij}$, and $\mc H$ be a Hilbert space with a $\zahl /2$-grading and a $\zahl /2$-graded $\Cliff _n$-action $c$. Then the (non-locally compact) space of odd bounded self-adjoint Fredholm operators $T$ that commute with the $\Cliff _n$-action and such that $c(e_1) \cdots c(e_n)T|_{\mc{H}^0}$ is neither positive nor negative definite modulo compact operators if $n$ is odd represents the $K^{-n}$-functor. 

Similarly, we represent the $K^n$-functor for $n > 0$ as some space of Fredholm operators. For an ungraded separable infinite dimensional Hilbert space $\mc{H}$, let $\mc{H}_{\Cliff _n}$ be a $\zahl /2$-graded Hilbert $\Cliff _n$-module $\mc{H} \hat{\otimes} \Cliff_n$. Now for $n>0$, let $\mc F_{\Cliff _n}(\mc H)$ be the (non-locally compact) space of odd bounded self-adjoint operators in $\mbb{B}(\mc{H}_{\Cliff _n})$ that is Fredholm, that is, invertible modulo $\mbb{K}(\mc{H}_{\Cliff _n})$. Moreover, if $n$ is odd, we additionally assume that $c(e_1) \cdots c(e_n)T|_{\mc{H}\otimes \Cliff_n^0}$ is neither positive nor negative definite. Then it represents the $K^n$-functor. It can be understood from the viewpoint of Kasparov's $KK$-theory (or bivariant $K$-theory) \cite{Kasparov1980}. As is well-known, the $KK$-theory has various formulations and the original one of Kasparov is deeply related to the theory of Fredholm operators and their indices (see also \cite{Blackadar1998}). For separable $\zahl /2$-graded $C^*$-algebras $A$ and $B$, a cycle in $KK(A,B)$ is of the form $[E,\varphi , F]$ where $E$ is a countably generated $\zahl /2$-graded Hilbert $B$-module, $\varphi $ is a $*$-homomorphism from $A$ to $\mbb{B}(E)$, and $F$ is an odd self-adjoint `Fredholm' operator on $E$ relative to $A$. More precisely, $F$ is an operator in $\mbb{B}(E)$ that satisfies $[\varphi(a) ,F]$, $\varphi(a)(F^2-1)$, and $\varphi(a)(F-F^*)$ are in $\mbb{K}(E)$ for any $a \in A$. 
A continuous family (in the norm topology) of $\Cliff _n$-equivariant odd Fredholm operators $F(x)$ ($x \in X$) gives a cycle $[\mc{H}_{\Cliff _n} \hat{\otimes} C(X) , 1 , F]$ in $KK(\comp , C(X) \hat{\otimes } \Cliff _n )$ by regarding $F$ as an element in $\mbb{B}(\mc H_{\Cliff_n} \otimes C(X))$ by pointwise multiplication. Because this $KK$-cycle depends only on its homotopy class, 
this correspondence gives a map from $[X,\mc F_{\Cliff _n}(\mc{H})]$ to $KK(\comp, C_0(X) \hat{\otimes} \Cliff _n)$. We can see that it is actually an isomorphism by using the equivalence relations called the operator homotopy \cite{Kasparov1980}. Here we do not have to care for addition of degenerate cycles by virtue of the Kasparov stabilization theorem \cite{Kasparov1980b}.

Now we have shown that there is some operator-theoretic description of the connective $K$-theory, but it is not consistent to the Fredholm picture of $KK$-theory and our construction of the $K$-theory spectrum. Next we see that these two are canonically related. 

Both of the two groups $KK(C_0(\real ^n),C(X))$ and $KK(\comp , C(X) \hat{\otimes} \Cliff _n)$ are isomorphic to $K^n(X)$. The canonical isomorphism $KK(C_0(\real ^n),C(X)) \to KK(\comp , C(X) \hat{\otimes} \Cliff _n)$ is given by taking the Kasparov product \cite{Kasparov1980} with the canonical generator of $KK(\comp, C_0(\real ^n) \otimes \Cliff _n)$ from the left. It also has many identifications and here we use the one in \cite{Kasparov1980}. It bases on the Fredholm picture and is of the form $[C_0(\real ^n) \hat{\otimes} \Cliff _n , 1 , C]$ where $C:= \sum c_i x_i (1+|x|^2) ^{-1/2}$. Here $c_i:=c(e_i)$ is the left multiplication of $e_i$ on $\Cliff _n$ that is a $\Cliff _n$-module by the right multiplication. 

Now we apply it for cycles that come from $\varphi \in \Hom (C_0(\real ^n),C(X) \otimes \mbb{K})$. We then have
\ma{&[C_0(\real ^n)\hat{\otimes} \Cliff _n , 1 , C] \otimes_{C_0(\real ^n)} [\mc{H} \hat{\otimes} C(X), \varphi , 0]\\
=&\lbk{C_0(\real ^n) \otimes _{\varphi} (\mc{H} \otimes C(X)) \hat{\otimes} \Cliff _n , 1 , C \otimes _{\varphi} \id }\\
=&\lbk{\mc{E}(\varphi) \hat{\otimes} \Cliff _n, 1 , \sum c_i T_i}.
}
Here we denote by $\mc{E}(\varphi)$ a Hilbert $C(X)$-module $\mbk{\overline{\varphi_x(C_0(\real ^n))\mc{H}}}_{x \in X}$ (more precisely, a subspace of $C(X) \otimes \mc{H}$ that consists of $\mc{H}$-valued functions on $X$ whose evaluations at $x$ are in $\overline{\varphi _x(C_0(\real ^n)\mc{H})}$). A $*$-homomorphism $\varphi : C_0(\real ^n) \to \mbb{B}(\mc{E}(\varphi))$ uniquely extends to $\tilde \varphi :C_b(\real ^n) \to \mbb{B}(\mc{E}(\varphi))$ because $\varphi $ is nondegenerate onto $\mbb{B}(\mc{E}(\varphi))$ (see Section 5 of \cite{Lance1995}). It is defined by the spectral measure and Borel functional calculus on each $\mc{H}_x$. We give $T_i := \tilde{\varphi} (x_i(1+|x|^2)^{-1/2})$.

This can be regarded as the Fredholm picture of connective $K$-theory. However, unfortunately it is not useful for our purpose because $\mc{E}(\varphi)$ may not be locally trivial and hence not a bundle of Hilbert spaces in general. Nonetheless, if a cycle in $\tilde{k}^n(X)$ has a good origin, then we have a better description for it. Actually cycles that arise in geometric contexts have that good origin and they are of our main interest.

\begin{defn}\label{def:tuple}
\item[$\bullet$] An $n$-tuple of bounded self-adjoint operators $(T_1, \ldots, T_n)$ on $\mc H$ is called a {\it bounded commuting Fredholm $n$-tuple} if it satisfies the following.
	\begin{enumerate}
	\item The operator $T^2:=\sum T_i^2$ is in $1+ \mbb{K}(\mc{H})$.
	\item The operator $T_i$ commutes with $T_j$ for any $i$ and $j$.
	\end{enumerate}
We denote by $\mc{F}_n(\mc{H})$ the set of bounded commuting Fredholm $n$-tuples equipped with the norm topology.
\item[$\bullet$] An $n$-tuple of unbounded self-adjoint operators $(D_1, \ldots, D_n)$ on $\mc H$ is an {\it unbounded commuting Fredholm $n$-tuple} if it satisfies the following.
	\begin{enumerate}
	\item The operator $D^2:=\sum D_i^2$ is densely defined, Fredholm, and has compact resolvents.
	\item The operator $D_i$ commutes with $D_j$ for any $i$ and $j$ on $\dom (D^2)^2$.
	\end{enumerate}
We denote the set of unbounded commuting Fredholm $n$-tuples by $\ms{F}_n(\mc{H})$. It is equipped with the strongest topology so that the map $(D^1,\ldots,D^n) \mapsto (D^1 (1+D^2)^{-1/2},\ldots,D^n(1+D^2)^{-1/2})$ is continuous. This definition is an analogue of the Riesz topology of the space of self-adjoint operators. 
\item For a bounded (resp. unbounded) commuting Fredholm $n$-tuple $(T_1 , \ldots , T_n)$ (resp. $(D_1 ,\ldots, D_n)$), we say that an odd self-adjoint operator $T:=c_1 T_1 + \cdots +c_n T_n$ on $\mc{H} \hat \otimes \Cliff _n$ (resp. $D:=c_1 D_1 + \cdots +c_nD_n$ with domain $\dom (D^2)^{1/2}$) is the {\it Dirac operator} associated with $(T_1,\ldots ,T_n)$. For simplicity of notation, hereafter we use the same letter $T$ (resp. $D$) for commuting Fredholm $n$-tuples and Dirac operators associated with it.
\end{defn}

The continuous map $(\overline{\mathbb{D}^n},\partial \mathbb{D}^n) \to (S^n,*)$ that collapses the boundary, more precisely of the form
$$(T_1,\ldots,T_n) \mapsto (2T^2-1,2(1-T^2)^{1/2}T_1,\ldots,2(1-T^2)^{1/2}T_n),$$
which is the unique continuous extension of composition map of the canonical isomorphism between $\mathbb{D}^n$ and $\real ^n$ and the stereographic projection, induces a continuous map $\iota : \mc{F}_n(\mc{H}) \to F_n(\mc{H})$ by functional calculus by definition of the topology on $\ms{F}_n(\mc{H})$. On the other hand, for $(T_1,\ldots,T_n) \in \mc{F}_n(\mc{H})$, the Dirac operator $T$ is in $\mc{F}_{\Cliff _n}(\mc{H})$. This correspondence gives a map from $[X,\mc{F}_n(\mc{H})]$ to $[X,\mc{F}_{\Cliff_n}(\mc{H})] \cong KK(\comp , C(X) \otimes \Cliff _n)$ that means, in a geometric context, to take the index bundle with $\Cliff _n$-module structure for the continuous family of Dirac operators associated with $(T_1 , \ldots , T_n)$. Hence we denote it by $\ind$.

\begin{thm}\label{thm:hom-op}
The following diagram commutes.
\[
\xymatrix{
[X,\mc{F}_n(\mc{H})] \ar[d]_\iota \ar[r]^{\ind \ \ \ \ \ \ \ } & KK(\comp , C(X)\hat \otimes \Cliff _n)  \\
[X,F_n(\mc{H})] \ar[r]^{\Phi \ \ \ \ \ \ \ } & KK(C_0(\real ^n) ,C(X))\ar[u]^{\rotatebox{90}{$\sim$}}.
}
\]
\end{thm}

\begin{proof}
Let $\mbk{T(x)}_{x \in X}:=\mbk{(T_1(x),\ldots,T_n(x))}_{x \in X}$ be a continuous family of bounded commuting Fredholm $n$-tuples and $\varphi ^{T}$ be its image by $\iota$. Then $\Phi \circ \iota [\mbk{T(x)}]$ is of the form $\lbk{\mc{E}(\varphi ^{T}) \hat \otimes \Cliff_n, 1, T}$. Now we give a homotopy connecting $[\ind T] =\lbk{(\mc{H} \otimes C(X)) \hat{\otimes} \Cliff _n,1, T(x)}$ and $\lbk{\mc{E}(\varphi ^{T}) \hat{\otimes} \Cliff_n, 1, T(x)}$ directly. It is given by a Kasparov $\comp$-$IC(X)$-bimodule 
$$\lbk{\mc{E}(\varphi ^{T} ) \oplus_{{\rm ev} _0} (\mc{H}_{C(X)} \otimes I), 1 , T}$$
where $\mc{E}(\varphi ^T ) \oplus_{{\rm ev} _0} (\mc{H}_{C(X)} \otimes I) :=\mbk{(x, f) \in \mc{E}(\varphi ^T ) \oplus (\mc{H}_{C(X)} \otimes I) \mid f(0)=x}$.
\end{proof}

\begin{remk}
For a general locally compact CW-complex we have an analogue of the $K$-theory with compact support. The $K$-group with compact support $K_{\rm cpt}^n(X)$ is defined as the the kernel of the canonical morphism $K^n(X^+) \to K^n(x_0)$ where $X^+$ is the one-point compactification of $X$ and $\mbk{x_0}=X^+ \setminus X$. It coincides with the $K$-group of the nonunital $C^*$-algebra $C_0(X)$ by definition. Similarly we write $k^n_{\rm cpt}(X)$ for the kernel of $k^n(X^+) \to k^n(x_0)$. When $X^+$ has a relatively compact deformation retract of $\mbk{x_0}$, $\tilde{k}^n_{\rm cpt}(X)$ is isomorphic to the homotopy class with compact support of maps from $X$ to $F(S^n,*)$ with compact support, $F_n(\mc{H})$, or $\Hom (C_0(\real ^n, \mbb{K}))$. Hence it is also isomorphic to $\Hom(C_0(\real ^n),C_0(X) \otimes \mbb{K})$. 
In terms of our Fredholm picture, a continuous family of Fredholm $n$-tuples on $X$ which is {\it bounded below} by some $\kappa >0$ (i.e. $D(x)^2 \geq \kappa$) outside some compact subset $K \subset X$ determines a $k^n$-cycle on $X$.
For simplicity we only denote $\tilde{k}(X)$ insted of $\tilde{k}_{\rm cpt}(X)$ in this paper.
\end{remk}

\begin{remk}
The formulation above is compatible with the product of cohomology theories. We define the product of continuous families of bounded commuting Fredholm $n$-tuples $T(x)=(T_1(x), \ldots , T_n(x))$ in $\Map (X, \mc{F}_n(\mc{H}))$ and $m$-tuples $S(x)=(S_1(x),\ldots,S_m(x))$ in $\Map (X, \mc{F}_m(\mc{H}'))$ as follows.
\ma{T(x)\times S(x)=&(T_1(x) , \ldots , T_n(x)) \times (S_1(x),\ldots,S_m(x)) \\
:= &(T_1(x) \otimes 1 ,\ldots, T_n(x) \otimes 1, 1 \otimes S_1(x),\ldots,1 \otimes S_m(x))\\
& \in \Map (X,\mc{F}_{n+m}(\mc{H} \otimes \mc{H}')).}
Then it is, up to homotopy, independent of the choice of $T(x)$ and $S(x)$. Consequently $[\mbk{T(x)}] \cup [\mbk{S(x)}] :=[\mbk{T(x) \times S(x)}]$ gives a well-defined product $[X, \mc{F}_n(\mc{H})] \times [X , \mc{F}_m(\mc{H})] \to [X,\mc{F}_{n+m}(\mc{H})]$ that is compatible with the product of connective $K$-groups, which is induced from the canonical map $(S^n,*) \times (S^m,*) \to (S^n ,*) \wedge (S^m,*) \cong (S^{n+m},*)$. By a similar argument we can define the product for unbounded commuting Fredholm $n$-tuples.
\end{remk}

\subsection*{twisted case}
Next, we generalize the above theory for the twisted connective $K$-theory. In the above argument, we have used the action of the Clifford algebra $\Cliff _n$ as the coefficients to construct a Dirac operator associated with a family of commuting Fredholm $n$-tuples. Now we regard it as the Clifford algebra bundle $\Cliff (\underline{\comp ^n})$ associated with the trivial bundle. We generalize the notion of the commuting Fredholm $n$-tuple and apply the general Clifford algebra bundles $\Cliff (V _\comp)$ associated with $Spin^c$ vector bundles $V$ for the coefficients of the Dirac operators associated with them.

We consider the canonical actions of $GL(n;\real)$ on the spaces $F(S^{n},*)$, $F_{n}(\mc{H})$, and $\Hom (C_0(\real ^{n}),C(X) \otimes \mbb{K})$. For example, on $F_{n+m}(\mc{H})$ it is of the form
$$g \cdot (T_0,T_1,\ldots ,T_{n}):=\bk{\sum g_{1j}T_j,\ldots,\sum g_{nj}T_j}.$$
Let $V$ be a real vector bundle over $X$. We denote a fiber bundle $GL(V) \times _{GL(n;\real)} F(S^n,*)$ (resp. $F_n(\mc{H})$) by $F_V$ (resp. $F_{V}(\mc{H})$). Similarly, $GL(n,\real )$ acts on the space of bounded (resp. unbounded) commuting Fredholm $n$-tuples $\mc{F}_{n}(\mc{H})$ (resp. $\ms{F}_{n}(\mc{H})$) and we denote by $\mc{F}_{V}(\mc{H})$ (resp. $\ms{F}_{V}(\mc{H})$) the corresponding fiber bundle.

\begin{defn}
A $V$-twisted family of bounded (resp. unbounded) commuting Fredholm $n$-tuples is a continuous section $T=T(x) \in \Gamma (X, \mc{F}_{V}(\mc{H}))$ (resp. $\Gamma (X , \ms {F}_{V}(\mc{H}))$).
\end{defn}

In the similar way to the above argument, the space of continuous sections $\Gamma \Cliff (V) =\Gamma (X, \Cliff(V))$ is a $C^*$-algebra and a continuous section $T \in \Gamma (X,\mc{F}_{V}(\mc{H}))$ defines a Kasparov $\comp$-$\Gamma \Cliff (V)$-bimodule
$$\lbk{\mc{H} \hat{\otimes} \Cliff (V), 1, c(e_1)T_{e_1}(x)+\cdots +c(e_n)T_{e_n}(x)},$$
which is independent of the choice of an orthonormal basis $\mbk{e_1,\ldots ,e_n} \in V_x$. Therefore we obtain a map $ \pi _0 (\Gamma (X,\mc{F})) \to KK(\comp , \Gamma \Cliff (V))$. 

Moreover, the following hold. 

\begin{prp}
Let $X$ be a finite CW-complex and $V$ a real vector bundle. The three sets
\begin{enumerate}
\item $\Gamma(X, F_V)$
\item $\Gamma(X, F_{V}(\mc{H}))$
\item $\Hom _{C(X)}(C_0(V) , C(X) \otimes \mbb{K})$
\end{enumerate}
are canonically mutually homeomorphic and their connected conponents form the twisted reduced connective $K$-group associated with the principal bundle $GL(V) \times _{GL(n,\real)}\mc{G}_k^{\rm mod}$, which we denote by $\tilde{k}^{V}(X)$ (see Section 3 of ~\cite{AtiyahSegal2004}). Here $\Hom _{C(X)}(C_0(V \times \real ^k) , C(X) \otimes \mbb{K})$ is the set of $C(X)$-homomorphisms, that is, $*$-homomorphisms that is compatible with their $C(X)$-module structures.
\end{prp}

\begin{thm}
Let $X$ be a finite CW-complex. Then the following diagram commutes.
\[
\xymatrix{
\pi _0( \Gamma (X,\mc{F}_{V}(\mc{H}))) \ar[d]_\iota \ar[r]^{\ind \ \ \ \ \ \ \ } & KK(\comp , \Gamma \Cliff(V))  \\
\pi _0 (\Gamma (X,F_{V}(\mc{H}))) \ar[r]^{\Phi \ \ \ \ \ \ \ } & \mc{R}KK(X;C_0(V) ,C(X))\ar[u]^{\rotatebox{90}{$\sim$}}.
}
\]
Here $\mc{R}KK(X;C_0(V),C(X))$ is the representable $KK$-group~\cite{Kasparov1988}. 
\end{thm}

In the same as the case of the $K$-theory, the Thom isomorphism holds for the twisted connective $K$-theory. 
\begin{prp}
The following isomorphism holds.
$$k^{W} (X) \cong k^{\pi^*V \oplus \pi ^*W}(V)$$
\end{prp}
\begin{proof}
Let $F$ be a closed subspace of $X$ and denote by $V_F$ the restriction $V|_F$ of vector bundles $V$. Then there is a morphism
\ma{\Hom _{C(F)} (C_0(W_F),C(F)\otimes \mbb{K}) &\ral \Hom_{C_0(V_F)} (C_0(\pi^*(V \oplus W)_{V_F}),C_0(V_F)\otimes \mbb{K})\\
\varphi &\longmapsto \id _{V} \otimes \varphi ,}
which is isomorphic if $V$ is trivial on $F$, and functorial with respect to inclusions. The Mayer-Vietoris exact sequence implies the global isomorphism.
\end{proof}
In particular, combining with the Thom isomorphism of the connective $K$-theory, we obtain the fact that the twist associated with $V$ is trivial if $V$ has a $Spin^c$ structure.

\section{The joint spectral flow}\label{section:3}
Now we give the precise definition of the joint spectral flow by using the notions introduced in Section \ref{section:2}. Next we prove an index theorem that generalizes the spectral flow index theorem of Atiyah-Patodi-Singer~\cite{AtiyahPatodiSinger1976}. Finally we generalize it for the case in which coefficients $c_i$ are globally twisted by a $Spin^c$ vector bundle. 

\subsection{Definitions and an index theorem}\label{section:3.1}

In the previous section we have seen that $F(S^n,*)$ represents the connective $K$-theory. Now we introduce another configuration space $P(X,A)$ with labels in positive integers on $X$ relative to $A$. More precisely, an element of $P(X,A)$ is a pair $(S,\mbk{n_x}_{x \in S})$ where $S$ is a countable subset of $X \setminus A$ whose cluster points are all in $A$ and each $n_x$ is a positive integer. Its topology is introduced in the same way as that of $F(X,A)$. Then $P(S^n,*)$ is canonically homotopy equivalent to the infinite symmetric product of $(S^n,*)$ that is a model of the Eilenberg-Maclane space $K(\zahl , n)$ by virtue of the Dold-Thom theorem~\cite{DoldThom1958}. There is a canonical continuous map $j$ from $F(S^n,*)$ to $P(S^n,*)$ ``forgetting'' data about vector spaces except for their dimensions, which is more precisely given by
$$(S,\mbk{V_x}_{x \in S}) \longmapsto (S,\mbk{\dim V_x}_{x \in S}).$$
In the viewpoint of commuting Fredholm $n$-tuples it forgets their eigenspaces and keeps only their joint spectra with multiplicity. 
It induces a group homomorphism
$$j_*: \tilde{k}^n(X) \ral H^n(X; \zahl).$$
Now we introduce the notion of the joint spectral flow.

\begin{defn}\label{def:jsf}
Let $X$ be an oriented closed manifold of dimension $n$. For a continuous family $\mbk{T(x)}=\mbk{(T_0(x),\ldots ,T_n(x))}_{x \in X}$ of elements in $F^n(\mc{H})$ parametrized by $X$, we say that $\ebk{j_* [\mbk{T(x)}] , [X]} \in \zahl$ is its {\it joint spectral flow} and denote it by $\jsf(\mbk{T(x)})$. For a continuous family of bounded (resp. unbounded) commuting Fredholm $n$-tuple $\mbk{T_1,\ldots,T_n}$, we say $\jsf(\iota \mbk{T(x)})$ is its joint spectral flow and denote it simply by $\jsf(\mbk{T(x)})$. 
\end{defn}

\begin{exmp}[the case of $n=1$]
According to Section 7 of \cite{AtiyahPatodiSinger1976}, the spectral flow is defined as the canonical group isomorphism ${\rm sf} : \pi _1 (F_1(\mc{H})) \to \zahl$ as follows. For a continuous map $T : S^1 \to F_1(\mc{H})$ whose essential spectrum is $\mbk{-1,1}$, there is a family of continuous functions $j_i:[0,1] \to [-1,1]$ such that $-1 =j=0 \leq j_1 \leq \cdots \leq j_m =1$ and $\sigma (T(t))=\mbk{j_0(t),\ldots,j_m(t)}$ for any $t \in [0,1]$. Then we obtain the integer $l$ such that $j_k(1)=j_{k+l}(0)$ for any $k$. This $l$ is called the spectral flow. Now let $\mbk{T(t)}$ be a continuous family of bounded self-adjoint Fredholm operators such that $\sigma (T(t))=\mbk{0,(t+1)/2,1}$ and the eigenspace $E_{(t+1)/2}$ is of dimension $1$. Then by definition its spectral flow ${\rm sf} (\mbk{T(t)})$ is equal to $1$. On the other hand, we obtain $j_*(\mbk{T(t)})=1 \in H^1(S^1 ;\zahl)$ since the canonical inclusion $S^1 \to \Sym ^\infty (S^1,*)$ gives a generator $1 \in H^1(S^1;\zahl) \cong [S^1, \Sym ^\infty (S^1,*)]$(see \cite{DoldThom1958} or Proposition5.2.23 of \cite{AguilarGitlerPrieto2002}). It means that the joint spectral flow coincides with the ordinary spectral flow in the case of $X=S^1$.
\end{exmp}

\begin{prp}\label{prp:nat_j}
The homomorphism $j_*$ is a natural transform of multiplicative cohomology theories.
\end{prp}
\begin{proof}
According to Section 3 of D{\u{a}}d{\u{a}}rlat-N{\'e}methi~\cite{DadarlatNemethi1990}, 
\ma{S: \Hom (C_0(\real ^n) , \mbb{K}) &\ral \Hom (C_0(\real ^{n+1}), C_0(\real) \otimes \mbb{K})\\
\varphi &\longmapsto \id _{\real} \otimes \varphi}
or equivalently
\ma{S: F(S^n,*) &\ral \Omega F(S^n \times I,S^n \times \mbk{0,1} \cup \mbk{*} \times I)\\
(S,\mbk{V_x}_{x \in S})& \longmapsto \mbk{t \mapsto ((x,t) ,\mbk{V_x}_{x \in S})}}
gives a homotopy inverse of $\Omega F(S^{n+1} ,*) \to F(S^n,*)$. By the same argument we obtain
\ma{S: P(S^n,*) &\ral \Omega P(S^n \times I,S^n \times \mbk{0,1} \cup \mbk{*} \times I)\\
(S,\mbk{n_x}_{x \in S})& \longmapsto \mbk{t \mapsto ((x,t) ,\mbk{n_x}_{x \in S})}}
gives a homotopy inverse of $\Omega P(S^{n+1}) \to P(S^n,*)$. 
Now by definition the following diagram commutes
\[
\xymatrix{F(S^n,*) \ar[r]^{S \ \ } \ar[d]_j &\Omega F(S^{n+1},*) \ar[d]^j \\
P(S^n,*) \ar[r]^{S \ \ } & \Omega P(S^{n+1},*).}
\]
The multiplicativity of $j_*$ follows immediately since the multiplicative structure on $\mbk{F(S^n,*)}_{n=0,1,2,\ldots}$ and $P(S^n,*)_{n=0,1,2,\ldots}$ are induced from the map $(S^n,*) \times (S^m,*) \to (S^{n+m},*)$ coming from the wedge product.
\end{proof}

To prove the generalization of the spectral flow index theorem, we will see the relation between the joint spectral flow and the Chern character. The Chern character is a natural transform from the $K$-functor to the rational cohomology functor. Here there is a generalization of the Chern character for a general cohomology theory, which was introduced by Dold~\cite{Dold1962} and is called the Chern-Dold character. 

Now we identify $k^*(X)$ with $\tilde{k}^{*+1}(SX)$ to extend $j_*$ to a natural transform between unreduced cohomology theories $k^*(X) \to H^*(X)$. It is compatible with the original $j_*$ according to Proposition \ref{prp:nat_j}.

\begin{prp}\label{prp:ch}
The $n$-th Chern-Dold character $\ch_n : k^n(X)\otimes \quot \to H^n(X; \quot)$ coincides with $j_*$ rationally. 
\end{prp}

\begin{proof}
The following diagram 
\[
\xymatrix{k^n(X) \otimes \quot \ar[r]_{\ch \ \ \ } \ar[d]_{j_*} \ar@{}[rd]|\circlearrowleft & H^n(X;k^* (\pt) \otimes \quot) \ar[d]_{1 \otimes j_*} \\
H^n(X ; \quot) \ar[r] ^{\sim \ \ \ \ } _{\ch =\id \ \ \ \ \ \ } & H^n(X; H^* (\pt) \otimes \quot)}
\]
commutes by Proposition \ref{prp:nat_j} and naturality of the Chern-Dold character. In fact, Dold proved in \cite{Dold1962} that there is a one-to-one correspondence between natural transforms of multiplicative cohomology theories $h \to h'$ and graded ring homomorphisms $h(\pt) \to h'(\pt)$ if $h'(\pt)$ is a graded vector space over $\quot$. The Chern-Dold character is induced from the ring homomorphism $h^*(\pt) \to \quot \otimes _\zahl h^*(\pt)$. Its naturality follows from the uniqueness. 

Now $k^*(pt) \cong \zahl[\beta ]$ ($\beta $ is of degree $-2$), $H^*(pt) \cong \zahl$, and ring homomorphism $j_*$ from $\zahl[\beta]$ to $\zahl$ is given by $1 \mapsto 1$ and $\beta \mapsto 0$. Hence $(1 \otimes j_*) \circ \ch$ coincides with the $n$-th Chern-Dold character $ch _n$. This implies that $j_*=\ch _n$. 
\end{proof}

Let $X$ be a closed $Spin^c$ manifold, $\slashed{\mf{S}}_\comp(X)$ the associated $\Cliff_n$-module bundle of $Spin^c(X)$ by the left multiplication on $\Cliff _n$ as right $\Cliff_n$-module, and $\slashed{\mf{D}}_X$ the $\Cliff _n$-Dirac operator on $\slashed{\mf{S}}_\comp(X)$. Now $\slashed{\mf{S}}_\comp (X)$ is equipped with the canonical $\zahl /2$-grading and $\slashed{\mf{D}}_X$ is an odd operator. Then it gives an element of $K_n(X) \cong KK(C(X) \hat{\otimes} \Cliff _n , \comp)$
$$[\slashed{\mf{D}}_X]:=[L^2(X,\slashed{\mf{S}}_\comp(X)), m, \slashed{\mf{D}}_X(1+\slashed{\mf{D}}_X^2)^{-1/2}],$$
which is the fundamental class of $K$-theory. Here $m: C(X) \hat \otimes  \Cliff _n \to B(L^2(\slashed{\mf{S}}_\comp(X)))$ is given by the Clifford multiplication. 

\begin{lem}\label{lem:pair}
Let $\mbk{T(x)}_{x \in X}$ be a continuous family of commuting Fredholm $n$-tuples. Then
$$\ebk{ [\ind T ],[\slashed{\mf{D}}_X]}_n= \jsf \mbk{T(x)}.$$
Here $\ebk{\cdot , \cdot}_n$ in the left hand side is the canonical pairing between $K^n(X)$ and $K_n(X)$.
\end{lem}

\begin{proof}
First we prove it in the case that $n$ is even. In that case we have a unique irreducible representation $\Delta_n$ of $\Cliff _n$ and the Dirac operator $\slashed{D}_X$ on $\slashed{S}_\comp (X):=Spin ^c(X) \times _{\Cliff _n} \Delta_n$. Now $\Delta _n$ is equipped with a canonical $\zahl/2$-grading and $\slashed{D}$ is an odd operator. It defines a $KK$-cycle
$$[\slashed{D}_X]:=[L^2(X,\slashed{S}_\comp (X)), m, \slashed{D}_X (1+\slashed{D}_X^2)^{-1/2}] \in KK(C(X) , \comp).$$
We denote by $[\![ \ind T ]\!]$ a $KK$-cycle $[\mc{H} \otimes \Delta_n , 1, T] \in KK(\comp ,C(X))$. Since $\Cliff _n \cong \Delta_n \otimes \Delta_n ^*$ as $\Cliff_n$-$\Cliff _n$-bimodules, the equalities $[\slashed{\mf{D}}_X]=[\slashed{D}_X] \otimes \Delta_n$ and $[\ind T]=[\![ \ind T ]\!] \hat \otimes \Delta _n$ (in particular $\ch [\ind T]=\ch [\![ \ind T ]\!]$) hold. Here $\Delta ^*$ is a Hilbert $\Cliff_n$-module by the inner product $\ebk{x,y}:=x^*y$. 

The pairing $\ebk{\cdot , \cdot}_n$ is given by the Kasparov product $KK(\comp , C(X) \otimes \Cliff _n) \otimes KK(C(X) \otimes \Cliff _n , \comp) \to \zahl$. Therefore
\ma{\ebk{[\ind T], [\slashed{\mf{D}}]}_n&=[\ind T] \otimes _{C(X) \otimes \Cliff _n}[\slashed{\mf{D}}_X] \\&=([\![ \ind T ]\!] \otimes _{C(X)} [\slashed{D}_X]) \otimes (\Delta ^* \otimes _{\Cliff _n} \Delta ) = [\![ \ind T ]\!] \otimes _{C(X)} [\slashed{D}_X].}
Now we use the Chern character for $K$-homology that is compatible with pairing. The Chern character of the $Spin ^c$ Dirac operator $\slashed{D}_X$ is given by its Todd class that is given by its $Spin^c$ structure. Hence
\ma{\ebk{[\mbk{T(x)}],[\slashed{D}_X]}&=\ebk{\ch ([\![ \ind T]\!]), \ch ([\slashed{D}_X])}\\
&=\ebk{\ch ([\ind T ]),\Td (X) \cap [X]}\\
&=\ebk{\ch _n ([ \ind T ] ) ,[X]}=\jsf \mbk{T(x)}.}
Here the third equality holds because $\ch ([ \ind T ])$ is in $\bigoplus _{k \geq 0} H^{n+2k}(X;\quot ) \cong H^n(X; \quot )$ and the zeroth Todd class $\Td _0 (X)$ is equal to 1 and the last equality holds by Proposition \ref{prp:ch}.

Finally we prove it in the case that $n$ is odd. We can reduce the problem to the case $n=1$ because for a family of self-adjoint operators $S(t)$ parametrized by $S^1$ whose spectral flow is $1$ (hence $[\ind S]=1 \in K^1(S^1) \cong \zahl$), we have
\ma{\ebk{[\ind T],[\slashed{\mf{D}}]}_n&=\ebk{[\ind T] \cup [\ind S] , [\slashed{\mf{D}}_X] \otimes [\slashed{\mf{D}}_{S^1}]}_{n+1}\\
&=\jsf (\mbk{T(x)} \times \mbk{S(t)})=\jsf \mbk{T(x)}.
}
Here we use the fact that the joint spectral flow of the product family $\mbk{T(x)} \times \mbk{S(t)}$ coincides with the product $\jsf(\mbk{T(x)}) \cdot \jsf(\mbk{S(t)})$.
\end{proof}

Now we give an index theorem that is a generalization of the spectral flow index theorem in \cite{AtiyahPatodiSinger1976}.

Let $B$ be a closed $n$-dimensional $Spin ^c$ manifold, $Z \to M \to B$ a smooth fiber bundle over $B$, $E$ a smooth complex vector bundle over $M$. We fix a decomposition $TM=T_VM \oplus T_HM$ of the tangent bundle where $T_VM:=\mbk{v \in TM; \pi _*v=0}$ is the vertical tangent bundle. For a hermitian vector bundle $E$, we denote by $\pi^*\slashed{\mf{S}}^E_\comp (B)$ the $\Cliff _n$-module bundle $\pi ^* \slashed{\mf{S}}_\comp (B) \otimes E$ on $M$. Now we define the {\it pull-back} of the $\Cliff_n$-Dirac operator $\slashed{\mf{D}}_B$ on $B$ twisted by $E$ as
\ma{\pi^* \slashed{\mf{D}}_B : &\Gamma (M,\pi ^* \slashed{\mf{S}}_\comp ^E(B)) \xra{\nabla} \Gamma (M , \pi ^*\slashed{\mf{S}} _\comp ^E(B)  \otimes T^*M)  \\
 & \hspace{3em} \xra{p_{T_H^*M}}\Gamma (M,\pi^*\slashed{\mf{S}} _\comp^E (B)\otimes T_H^*M) \xra{h}\Gamma (M,\pi ^*\slashed{\mf{S}}_\comp ^E(B)).}
Here, $h$ is the left Clifford action of $\Cliff (TB) \cong \Cliff (T_HM)$ on $\pi^*\slashed{\mf{S}}_\comp^E(B)$. We write down it by using an orthogonal basis $\mbk{e_1,\ldots,e_n }$ of $T_{\pi(x)}B \cong T_{\pi(x)}^*B$ as 
$$\pi ^*\slashed{\mf{D}}_B= \sum h(\pi ^* e_i) \nabla ^{\pi ^*\slashed{\mf{S}}_\comp ^E(B)}_{\pi^* e_i}.$$
Now it satisfies 
\ma{\pi ^* \slashed{\mf{D}}_B (\pi ^* \varphi )=\pi ^* (\slashed{\mf{D}}_B\varphi).}

Let $\mbk{D_1,\ldots,D_n}$ be an $n$-tuple of fiberwise first order pseudodifferential operators on $E$, that is, a smooth family $\mbk{D(x)}$ of pseudodifferential operators on $\Gamma (M_x , E|_{M_x})$. Moreover we assume these two conditions.
\theoremstyle{definition}
\newtheorem{cond}[equation]{Condition}
\begin{cond}\label{cond:comm}
\item[1.] The operators $D_i$ and $D_j$ commute for any $i,j$.
\item[2.] The square sum $\sum _{i=1}^n D_i ^2$ is fiberwise elliptic, that is, its principal symbol is invertible on $S(T_VM)$.
\end{cond}
Then, by taking a trivialization of the Hilbert bundle of fiberwise $L^2$-sections $\mc{L}^2_f(M,E \hat \otimes \Cliff _n):=\mbk{L^2(Z_x,E_x \hat \otimes \Cliff _n)}_{x \in B}$, it forms a continuous family of unbounded commuting Fredholm $n$-tuples $\mbk{D(x)}=\mbk{(D_1(x),\ldots,D_n(x))}$ parametrized by $B$. Indeed, according to Kuiper's theorem, any Hilbert space bundles are trivial and $[D(x)]$ is independent of the choice of a trivialization. The second assertion holds because a trivialization of Hilbert bundle $\mc{V}$ gives a unitary $U \in \Hom _{C(X)} (C(X) \otimes \mc{H}, \Gamma (X, \mc{V}))$ and hence two trivializations $U$, $U'$ give a norm continuous unitary-valued function $U^{-1}U'$, which is homotopic to the identity. Combining with a connection on $\pi^*\slashed{\mf{S}}_\comp (B)$, which is fiberwise flat, the Dirac operator $D(x)=c_1D_1(x) + \cdots +c_nD_n(x)$ associated with $\mbk{D(x)}$ (here we denote by $c$ the $\Cliff_n$-action on $\slashed{\mf{S}}_\comp (B)$ and $c_i:=c(e_i)$ for an orthonormal basis $\mbk{e_i}$) also defines a first order pseudodifferential operator on $\pi ^*\slashed{\mf{S}}^E_\comp (B)$.

Now we describe our main theorem.

\begin{thm}\label{thm:jsf}
Let $B$, $M$, $E$, and $\mbk{D(x)}$ be as above. Then the following formula holds.
\ma{\ind _0 (\pi ^*\slashed{\mf{D}}_B +D(x))=\jsf \mbk{D(x)}.}
Here, for an odd self-adjoint operator $D$ on $\mc{H} =\mc{H}^0 \oplus \mc{H}^1$, we denote by $\ind _0 D$ the Fredholm index of $D : \mc{H}^0 \to \mc{H}^1$.
\end{thm}

To prove this theorem we prepare a lemma about an operator inequality. In this section we denote $D(x)$ and $\pi ^*\slashed{\mf{D}}_B$ simply by $D_f$ and $D_b$. 

\begin{lem}\label{lem:ineq}
For any $\alpha  \geq 0$ there is a constant $C>0$ such that for any $\xi \in \Gamma (M,\pi^*\slashed{\mf{S}}_\comp^E(B))$
\maa{\ebk{[D_b,D_f]\xi,\xi} \geq -\alpha \ssbk{D_f \xi}^2-C\ssbk{\xi}^2. \label{eq:ineq}}
\end{lem}

\begin{proof}
First we observe that $[D_b,D_f]$ is also a fiberwise first-order pseudodifferential operator. Let $(V, x_b^1,\ldots,x_b^n)$ be a local coordinate of $x \in B$ and $(U, x_b^1,\ldots, x_b^n,x_f^1,\ldots ,x_f^m)$ a local coordinate in $\pi^{-1}(V)$ such that tangent vectors $\partial _{x_b^i}(p)$ are in $(T_HM)_p$ for any $p \in {\pi^{-1}(x)}$. We get such a coordinate by identifying a neighborhood of zero section of $T_HM|_{\pi^{-1}(x)} \cong N\pi^{-1}(x)$ with a tubular neighborhood of $\pi^{-1}(x)$. We assume that $\pi^*\slashed{\mf{S}}_\comp^E(B)$ is trivial on $U$ and fix a trivialization. Then, for any fiberwise pseudodifferential operator $P$ supported in $U$, the operator $[\partial _{x_b^i},P]$ is also fiberwise pseudodifferential. Indeed, when we write down a fiberwise pseudodifferential operator $P$ on a bounded open subset of $\real^{n+m}=\real ^n_{x_b}\times \real ^m_{x_f}$ as
$$Pu(x_b,x_f)=\int _{(y_f,\xi_f) \in \real ^m \times \real ^m} e^{i\ebk{x_f-y_f,\xi_f}}a(x_b,x_f,y_f,\xi _f)u(x_b,y_f)dy_fd\xi_f,$$
we have
\ma{\lbk{\partial_{ x_b^{i}},P}u(x_b,x_f)=&\int \partial _{x_b^{i}}(e^{i\ebk{x_f-y_f,\xi_f}}a(x_b,x_f,y_f,\xi_f)u(x_b,y_f))dy_fd\xi_f \\
&-\int e^{i\ebk{x_f-y_f,\xi_f}}a(x_b,x_f,y_f,\xi_f)\partial _{x_b^i}u(x_b,y_f)dy_fd\xi_f\\
=&\int e^{i\ebk{x_f-y_f,\xi_f}}(\partial _{x_b^i}(a(x_b,x_f,y_f,\xi_f))u(x_b,y_f)dy_fd\xi_f .}
Let $D'_b:=\sum g^{ij}h(\partial _{x_b^i})\nabla _{\partial _{x^j_b}}$. Since the Riemannian metric $g^{ij}$ on $T_HM$ only depends on the local coordinate of $B$ (i.e. is a function on $B$), an operator $[D'_b,P]=[\sum g^{ij}h(\partial _{x_b^i})(\partial _{x_b^j} +\omega(\partial_{x_b^j})),P]$ is also fiberwise pseudodifferential. 

For any $\xi \in \Gamma (U,\pi^*\slashed{\mf{S}}_\comp^E(B)|_U)$, the section $[D_b,P]\xi|_{\pi^{-1}(x)}$ depends only on the restriction of $\xi$ and its differentials in normal direction on $\pi^{-1}(x)$. Since the Dirac operator $D_b$ coincides with $D'_b$ on $U_0:=U \cap \pi^{-1}(x)$ and $[P,D'_b]$ is fiberwise pseudodifferential, $[D_b,P]\xi|_{\pi^{-1}(x)}=[D'_b,P]\xi|_{\pi^{-1}(x)}$ does not depend on differentials of $\xi$. Now, the above argument is independent of the choice of $x \in B$. As a consequence, $[D_b,P]$ is also fiberwise pseudodifferential. By using a partition of unity, we can see that $[D_b,D_f]$ is also a fiberwise pseudodifferential operator. 

As a conseqence, we obtain that $[D_b,D_f](1+D_f^2)^{-1/2}$ is a zeroth order pseudodifferential operator. In particular, it is bounded. Now, for any $\lambda >0$, we obtain an inequality
\ma{
\ebk{[D_b,D_f]\xi,\xi}&=\ebk{\lambda[D_b,D_f]\xi,\lambda^{-1}\xi} \\
&\ebk{\lambda D_b D_f(1+D_f^2)^{-1/2}(1+D_f^2)^{1/2}\xi,\lambda^{-1}\xi} \\
&+\ebk{\lambda D_f D_b\xi,\lambda^{-1}(1+D_f^2)^{-1/2}(1+D_f^2)^{1/2}\xi} \\
&\geq -\frac{1}{2} \lambda^{2}\ebk{[D_b,D_f]\xi,[D_b,D_f]\xi}-\frac{1}{2}\lambda^{-2}\ebk{\xi,\xi}\\
&\geq -\frac{1}{2}\lambda^{2}\ssbk{[D_b,D_f](1+D_f^2)^{-1/2}}^{2}\ebk{(1+D_f^2)\xi,\xi}- \frac{1}{2}\lambda^{-2}\ebk{\xi,\xi}\\
&=-\frac{1}{2}\lambda^{2}\ssbk{[D_b,D_f](1+D_f^{2})^{-1/2}}^{2}\ebk{D_f\xi,D_f\xi}\\
&- \frac{1}{2}(\lambda ^2\ssbk{[D_b,D_f](1+D_f^{2})^{-1/2}}^{2}+\lambda^{-2})\ebk{\xi,\xi}}
as is introduced in Lemma 7.5 of Kaad-Lesch~\cite{KaadLesch2012}. Now by choosing $\lambda := \sqrt{2 \alpha}\ssbk{[D_b,D_f](1+D_f^{2})^{-1/2}}^{-1}$ and $C:=\alpha + \lambda ^{-2} /2$, we show this $C$ satisfies the above condition. 
\end{proof}

Now we use the Connes-Skandalis type sufficient condition to realize the Kasparov product unbounded Kasparov bimodules introduced by Kucerovsky~\cite{Kucerovsky1997}.

\begin{thm}[Kucerovsky~\cite{Kucerovsky1997}]\label{thm:Kas}
Suppose that $(E_1,\varphi _1,D_1)$, $(E_2,\varphi _2,D_2)$, and $(E_1 \hat \otimes E_2,\varphi _1 \hat \otimes 1 , D)$ are unbounded Kasparov bimodules for $(A,B)$, $(B,C)$, and $(A,C)$ such that the following conditions hold.
\begin{enumerate}
\item For all $x$ in some dense subset of $\varphi _1(A)E_1$, the operator
$$\lbk{\pmx{D & 0 \\0 & D_2},\pmx{0 & T_x \\ T_x^* & 0}}$$
is bounded on $\dom (D \oplus D_2)$.
\item The resolvent of $D$ is compatible with $D_1 \hat{\otimes} 1$.
\item For all $x$ in the domain, $\ebk{D_1x,Dx} + \ebk{Dx,D_1x} \geq \kappa \ebk{x,x}$.
\end{enumerate}
Here $x \in E_1$ is homogeneous and $T_x : E_2 \to E$ maps $e \mapsto x \hat \otimes e$. Then $[E_1 \hat \otimes E_1,\varphi _1 \hat \otimes 1 , D] \in KK(A,C)$ represents the Kasparov product of $[E_1,\varphi _1,D_1] \in KK(A,C)$ and $[E_2,\varphi _2,D_2] \in KK(B,C)$.
\end{thm}
Here the resolvent of $D$ is said to be {\it compatible} with $D'$ if there is a dense submodule $\mc{W} \subset E_1 \hat \otimes E_2$ such that $D'(i\mu +D)^{-1}(i\mu ' +D')^{-1}$ is defined on $\mc{W}$ for any $\mu , \mu ' \in \real \setminus \mbk{0}$. It holds for example in the case that $\dom D \subset \dom D'$.

\begin{proof}[Proof of Theorem \ref{thm:jsf}]
According to Lemma \ref{lem:pair}, the remaining part for the proof is that the left hand side coincides with the pairing $\ebk{[\ind D], [\slashed{\mf{D}}_B]}_n$. Here this pairing is given by the Kasparov product $KK(\comp , C(B) \hat{\otimes } \Cliff _n) \otimes KK(C(B) \hat{\otimes} \Cliff _n, \comp) \to \zahl$. It is computed as follows. 
\ma{
&\lbk{\mc{L}^2(M,E \hat{\otimes} \Cliff _n) , 1 , D } \otimes _{C(B) \hat{\otimes} \Cliff _n} \lbk{L^2(B, \slashed{\mf{S}}_\comp(B)) , m , \slashed{\mf{D}}_B}\\
&=\lbk{L^2(M, (E \hat \otimes \Cliff _n) \hat {\otimes} _{\Cliff _n} \pi ^*\slashed{\mf{S}}_\comp (B)) , 1 , \slashed{\mf{D}}_B \times D}\\
&=\lbk{L^2(M,  \pi ^*\slashed{\mf{S}}_\comp (B)^ E) , 1 , \slashed{\mf{D}}_B \times D}.
} 
Now the rest to prove is that $D_b+D_f$ satisfies conditions 1, 2, and 3 of Theorem \ref{thm:Kas}. 

For any $\sigma \in C^\infty (M,E)$ and $\xi \in C^\infty (B,\slashed{\mf{S}}_\comp(B))$, the Leibniz rule of $\pi ^*\slashed{\mf{D}}_B$ implies that 
\ma{(D_b+D_f) T_\sigma \xi &= (D_b +D_f)(\sigma \cdot \pi^*\xi)=(D_b+D_f) x \cdot \pi^*\xi +\sigma \cdot D_b\pi ^*\xi \\
&=T_{(D_b+D_f)\sigma}\xi + \sigma \cdot \pi ^* (\slashed{\mf{D}}_B \xi).}
Therefore $(D_b+D_f)T_\sigma -T_\sigma \slashed{\mf{D}}_B=T_{(D_b+D_f)\sigma}$ is a bounded operator and hence condition 1 holds. Condition 2 holds since $\dom (D_b+D_f) \subset \dom D_f$. For any $\xi \in C^\infty(M,\slashed{\mf{S}}_\comp ^E(M))$, which is dense in the domain, 
\ma{\ebk{D_f\xi, (D_b+D_f)\xi} + \ebk{(D_b+D_f)\xi , D_f\xi}&=\ebk{[D_b,D_f]\xi ,\xi} +\ssbk{D_f\xi}^2.}
Condition 3 follows from it and Lemma \ref{lem:ineq}.
\end{proof}

\begin{remk}\label{rem:findim}
The calculus above is motivated by that of Connes-Skandalis~\cite{ConnesSkandalis1984}, in which they dealt with principal symbols and zeroth order pseudodifferential operators. Here we use the unbounded operators directly to apply it for more general cases. For example, by the same argument we obtain a similar formula
$$\ind _0(D+A(x))= \jsf(\mbk{A(x)})$$
for a smooth family of mutually commuting self-adjoint complex coefficient matrices $A(x)=(A_1(x),\ldots ,A_n(x))$. Other examples are given in the next section. 
\end{remk}

\subsection{A Callias type index theorem for open manifolds}\label{section:3.2}

Now we consider generalizing our index theorem for the case of noncompact base spaces. The pairing of homology and cohomology works in the noncompact case if the cohomology is replaced with the one with compact support. We can deal with it in the context of an infinite dimensional analogue of Callias-type operators \cite{Callias1978}. Here we use fiberwise elliptic operators as the potential term in the original theory of Callias. First we define the admissibility of a connective $K$-cocycle (see also \cite{Bunke1995}).

\begin{defn}
A continuous family of commuting Fredholm $n$-tuples $\mbk{D_1,\ldots,D_n}$ parametrized by a complete Riemannian manifold $B$ is said to be {\it admissible} if there is a constant $c>0$ that satisfies the following.
\begin{enumerate}
\item $D(x)^2  \geq \kappa >0 $ for  $x \in X \setminus K$, 
\item There are $C_1>0$ and $C_2>0$ such that $\ebk{([D_b,D_f]+D_f^2)\xi,\xi} \geq C_1\ssbk{D_f\xi}^2- C_2\ssbk{\xi}^2$ and $\kappa C_1>C_2$.
\end{enumerate}
\end{defn}

Actually the second condition is not essential. 

\begin{lem}
For any continuous family of commuting Fredholm $n$-tuples $\mbk{D_1,\ldots,D_n}$ parametrized by a complete $n$-dimensional Riemannian manifold $B$ that satisfies condition 1 above, there is some $t >0$ such that $tD:=(tD_1,\ldots,tD_n)$ is admissible.
\end{lem}

\begin{proof}
By a similar calculus to the one in Lemma \ref{lem:ineq} (we replace $D_f$ in the first term with $tD_f$ but do not replace the one that arises in $(1+D_f^2)$ in the middle part) we show that for any $\lambda >0$
\ma{
\ebk{[D_b,tD_f]\xi,\xi}&=-\frac{1}{2}\lambda^{2}R\ebk{D_f\xi,D_f\xi}- \frac{1}{2}(\lambda ^2R+\lambda^{-2})\ebk{\xi,\xi}.}
where we denote that $R:=\ssbk{[D_b,D_f](1+D_f^2)^{-1/2}}^2$. Now if we choose $\lambda =R^{-1/2}$, then 
$$\ebk{([D_b,tD_f]+(tD_f)^2) \xi,\xi} \geq \frac{t^2}{2} \ssbk{D_f\xi}^2 - \bk{\frac{t^2}{2} + R}  \ssbk{\xi}^2.$$
Now we can take a constant $\kappa$ in condition 1 for $tD_f$ as $t\kappa$. When we set $C_1=\frac{t^2}{2}$ and $C_2=\frac{t^2}{2}$, for sufficiently large $t>0$ the inequality $(t\kappa) C_1 \geq C_2 $ holds and hence the constants $t\kappa$, $C_1$, and $C_2$ satisfies condition 2.
\end{proof}

Now we introduce a geometric setting and an index Theorem for the noncompact case.

Let $B$ be a complete $n$-dimensional manifold, $Z \to M \to B$ a smooth fiber bundle over $B$ with fixed decomposition of the tangent bundle $TM \cong T_VM \oplus T_HM$, $E$ a smooth complex vector bundle over $M$ and $\mbk{D_1,\ldots,D_n}$ an $n$-tuple of fiberwise first order pseudodifferential operators on $E$ that satisfies the Condition \ref{cond:comm}. Moreover we assume that $\mbk{D_1,\ldots,D_n}$ is admissible. 

\begin{thm}\label{thm:jsfopen}
In the above situation, the operator $\pi ^*\slashed{\mf{D}}_B +D(x)$ is Fredholm and the following formula holds.
\ma{\ind _0(\pi^*\slashed{\mf{D}}_B +  D(x))=\jsf \mbk{D(x)}}
\end{thm}

\begin{proof}
The proof is essentially the same as for Theorem \ref{thm:jsf} and the remaining part is to show that $\slashed{\mf{D}}_B + D(x)$ is a Fredholm operator. We prove it by using an estimate motivated by Theorem 3.7 of Gromov-Lawson~\cite{GromovLawson1984}. Here we use the convention $D_b$ and $D_f$ again. Let $E_\lambda$ ($\lambda \in \real$) be the eigenspace for the self-adjoint operator $D_b+D_f$. Now we fix an $\alpha >0$. Then for any $\sigma \in \bigoplus _{|\lambda |<\alpha} E_\lambda $,
\ma{ 0 &\leq \ssbk{D_b\sigma }^2 \leq \ssbk{(D_b+D_f)\sigma }^2 - (([D_b,D_f]+D_f^2)\sigma ,\sigma )\\
&\leq \alpha \ssbk{\sigma }^2 - C_1 \ssbk{D_f\sigma }^2 +C_2 \ssbk{\sigma }^2\\
&\leq (\alpha +C_2) \ssbk{\sigma }^2 - C_1 \ssbk{D_f \sigma }^2 _{B \setminus K}\\
&\leq (\alpha -\kappa C_1 +C_2) \ssbk{\sigma }^2 + \kappa C_1 \ssbk{\sigma }_K^2.}
By assumption we can retake $\alpha >0$ such that $\kappa C_1-C_2 >\alpha$. Then there is a constant $C>0$ and we obtain an estimate
$$\ssbk{\sigma } \leq C\ssbk{\sigma }_K.$$
Now we take a parametrix $Q$ of the elliptic operator $D_b+D_f$ and $\mc{S}:=1-QD$. Let $P$ be the projection from $L^2(M, \pi ^* \slashed{\mf{S}}_\comp^E (B))$ to the subspace $L^2(\pi ^{-1}(K), \pi ^* \slashed{\mf{S}}_\comp^E (B)|_{\pi ^{-1(K)}})$. Then $P\mc{S}$ is a compact operator and 
$$\ssbk{P\mc{S}\sigma } \geq \ssbk{P\sigma}-\ssbk{PDQ\sigma} \geq C\ssbk{\sigma} -\alpha \ssbk{PQ} \ssbk{\sigma } .$$
Choosing $\alpha >0$ sufficiently small, we see that $P\mc{S}$ is bounded below by $C-\alpha \ssbk{PQ} >0$. It implies that $\bigoplus E_\lambda $ is finite dimensional since a compact operator on it is bounded below by some positive number.
\end{proof}

\begin{exmp}[the case of $B=\real$]
Let $\mbk{A(t)}_{t \in \real}$ be a continuous family of self-adjoint matrices such that there is a $\lambda >0$ and two self-adjoint matrices $A_+$, $A_-$ such that $A_t=A_-$ for $t \leq -\lambda $ and $A_t=A_+$ for $ \lambda \leq t$. Now as is noted in Remark \ref{rem:findim}, we have a finite dimensional analogue of Theorem \ref{thm:jsfopen}. In the $1$-dimensional case it is of the form
$$\ind (\frac{d}{dt} + A_t)={\rm sf} (\mbk{A_t}).$$
Now obviously its right hand side is given by the difference
$$\# \mbk{ \text{positive eigenvalues of $A_-$}}- \# \mbk{\text{negative eigenvalues of $A_+$}}.$$
It is nonzero in general whereas in the case that the parameter space is a circle we have to deal with operators on an infinite dimensional Hilbert space to obtain an example of nontrivial indices. 
\end{exmp}

\begin{exmp}
Let $B$ be a complete $Spin^c$ manifold, $Z_1, \ldots , Z_n$ be closed odd dimensional $Spin^c$ manifolds and $\mbk{g^1_x,\ldots ,g_x^n}_{x \in B}$ be a smooth family of metrics on $M_1 ,\ldots, M_n$ such that the scalar curvature of the product manifold $Z:=Z_1 \times \cdots \times Z_n$ is uniformly strictly positive outside a compact subset $K \subset B$. We denote by $\slashed{D}_{i,x}$ the Dirac operator on $Z_i$ with respect to the metric $g^i_x$. Then there is a constant $\lambda >0$ such that $(\lambda \slashed{D}_{1,x},\ldots,\lambda \slashed{D}_{n,x})$ is an admissible family of commuting Fredholm $n$-tuples and the Fredholm index of the $Spin^c$ Dirac operator on $M:=B \times Z$ with respect to the product metric coincides with its joint spectral flow. This gives a map
$$\ind : [(B^+,*),(\mc{R} (Z_1,\ldots,Z_n),\mc{R}(Z_1,\ldots,Z_n)_{\geq \lambda})] \to \zahl$$
where $\mc{R}(Z_1, \ldots , Z_n)$ is the product of spaces of Riemannian metrics $\mc{R}(Z_1) \times \cdots \mc{R}(Z_n)$ and $\mc{R}(Z_1,\ldots,Z_n)_{\geq \lambda}$ is the subspace of $\mc{R}(Z_1,\ldots,Z_n)$ such that the scalar curvature of the product metric $(Z_1,g_1) \times \cdots \times (Z_n,g_n)$ is larger than $\lambda >0$ (its homotopy type is independent of the choice of $\lambda$). 
In particular when we choose $B$ as $\real ^n$ the left hand side is isomorphic to $\pi _{n-1} (\mc{R}(Z_1,\ldots,Z_n)_{\geq \lambda})$ because $\mc{R}(Z_1,\ldots,Z_n)$ is contractible. 
\end{exmp}

\subsection{Families twisted by a vector bundle}
In this section we generalize the joint spectral flow and its index theorem for the case of $V$-twisted families of commuting Fredholm $n$-tuples introduced at the end of Section \ref{section:2}. It is essential in Section \ref{section:4.1}. 

Let $V$ be a real vector bundle. We denote by $P_{V}$ the fiber bundle $GL(V) \times _{GL(n,\real)} P(S^{n},*)$. The set of homotopy classes of continuous sections $\pi _0 \Gamma (X,P_{V})$ forms the twisted cohomology group $H^{V}(X;\zahl)$. Now, twists of the ordinary cohomology theory is classified by $H^1(X,\zahl /2)$ and in our case the corresponding cohomology classes are determined by the orient bundle of $V$. As is definition \ref{def:jsf}, there is the continuous map $j: F_{V}(\mc{H}) \to P_{V}$, which induces the natural transform $j_*:k^{V} \to H^{V}$.

\begin{defn}\label{def:jsftwisted}
Let $X$ be an oriented closed manifold of dimension $n$ and $V$ an $n$-dimensional oriented vector bundle. For a continuous family $\mbk{T(x)}_{x \in X}$ of commuting Fredholm $n$-tuple twisted by $V$, we say that the integer $\ebk{j_* [\mbk{T(x)}] , [X]} \in \zahl$ is its {\it joint spectral flow} and denote it by $\jsf(\mbk{T(x)})$. Here we identify two groups $H^V(X;\zahl)$ and $H^n(X;\zahl)$ in the canonical way. For a continuous family of bounded (resp. unbounded) commuting Fredholm $n$-tuple $\mbk{T(x)}$ twisted by $V$, we say $\jsf(\iota \mbk{T(x)})$ is its joint spectral flow and denote it simply by $\jsf(\mbk{T(x)})$. 
\end{defn}

Now we introduce the corresponding geometric setting and prove a generalization pf the joint spectral flow index theorem \ref{thm:jsf} for a family twisted by a $Spin^c$ vector bundle.

Let $B$ be a closed $n$-dimensional $Spin^c$ manifold, $Z \to M \to B$ a smooth fiber bundle over $B$ such that the total space $M$ is also a $Spin^c$ manifold, $V$ be an $n$-dimensional $Spin^c$ vector bundle over $B$, and $E$ a smooth complex vector bundle over $M$. We denote by $\Psi _f^1 (M,E)$ the fiber bundle over $B$ whose fiber on $x \in B$ is the space of first order pseudodifferential operators on $\Gamma (M_x,E|_{M_x})$. We consider a map of $B$-bundles $\mbk{D_v(x)}_{(x,v) \in V \setminus \mbk{0}}:V \setminus \mbk{0} \to \Psi ^1 _f (M,E)$ that satisfies the following conditions.

\begin{cond}\label{cond:commtwist}
\item[1.] The operators $D_v(x)$ and $D_w(x)$ commute for any $v,w \in V_x \setminus \mbk{0}$.
\item[2.] The equality $g \cdot (D_{v_1}(x),...,D_{v_n}(x))=(D_{g \cdot v_1}(x) , \ldots , D_{g \cdot v_n}(x))$ holds for any $g \in GL(n;\real)$ and a basis $(v_1,...,v_n) $ of $V_x$.
\item[3.] The square sum $\sum _{v_1,\ldots,v_n {\rm : \ ONB}} D_{v_i} ^2$ is fiberwise elliptic, that is, its principal symbol is invertible on $S(T_VM)$.
\end{cond}

Then it forms a continuous family of unbounded commuting Fredholm $n$-tuples $\mbk{D(x)}$ twisted by $V$. 

Next, we replace the fundamental $KK$-class on $B$ with the one that is compatible with $\mbk{D(x)}$. Instead of $\slashed{\mf{S}}_\comp (M)$, we consider the spinor bundle $\slashed{\mf{S}}_\comp(B;V):=\slashed{S}_\comp (TB \oplus V)$ for an even dimensional $Spin^c$ vector bundle $TB \oplus V$. It is equipped with the action of $\Cliff (TB) \hat \otimes \Cliff (V)$. Here we denote by $c $ and $h$ its restriction on $\Cliff(V) \hat \otimes 1$ and $1 \hat \otimes \Cliff (TB)$ respectively. Now we define a pull-back of the Dirac operator $\pi^*\slashed{\mf{D}}_B^V$ twisted by $E$ in a similar way to the one in Section \ref{section:3.1}. 

\begin{thm}\label{thm:jsftwisted}
Let $B$, $M$ and $D(x)$ be as above. Then the following formula holds.
\ma{\ind (\pi ^*\slashed{\mf{D}}_B^V + D(x))=\jsf \{ D(x) \} .}
\end{thm}

\begin{proof}
First we embed $V$ into a trivial real vector bundle $\underline{\real}^p$ linearly and denote its orthogonal complement by $W$. 

We define the following $KK$-classes
\ma{[D_W]&:=\lbk{\mc{L}_f^2(W ,\Cliff (\pi ^*W)), m , D_W:=\sum h(e_i) \frac{\partial}{\partial w_i}} \in KK(\Gamma _0 \Cliff (\pi^*W) , C(B)),\\
[C_W]&:=\lbk{\Gamma _0 \Cliff (\pi^*W) , m , C_W:=\sum c(e_i) w_i} \in KK(C(B),\Gamma _0 \Cliff (\pi^*W)),}
where $\mbk{e_i}$ is an orthonormal basis on $W_x$ and $w_i=\ebk{w,e_i}$ the coordinate functions with respect to $\mbk{e_i}$. We mention that $D_W$ and $C_W$ are independent of the choice of $\mbk{e_i}$ and hence they are well-defined. 
Then, the theory of harmonic oscillators (see for example Section 1.13 of \cite{HigsonGuentner2004}) shows that $[D_W] \otimes _{\Gamma _0 \Cliff (\pi ^*W)} [C_W]=[D_W+C_W] =1 \in KK(C(B) ,C(B))$ because the kernel of the harmonic oscillator is one dimensional and $O(n)$-invariant.
Now
$$D \times C_W=(D_{v_1},\ldots,D_{v_n},c_{w_1},\ldots,c_{w_k})$$
is a smooth family of commuting Fredholm $n$-tuples twisted by $V \oplus W \cong \underline{\real}^p$. Moreover it is admissible on $W$ because $(D \times C_W)^2=D^2+\ssbk{w}^2$. According to Theorem \ref{thm:jsfopen}, 
\ma{\ind (D_b+D_f+D_W+C_W)=\jsf (\mbk{D \otimes C_W(x,w)})=\jsf (\mbk{D(x)}).}
On the other hand, by the associativity of the Kasparov product 
\ma{\ind (D_b+D_f+D_W+C_W)&=[D_f+C_W]\otimes_{\Gamma_0(\pi^*W)} [D_W+D_b]\\
&=([D_f] \otimes _{C(B)} [C_W]) \otimes _{\Gamma_0(\pi^*W)}([D_W] \otimes _{C(B)} [D_f])\\
&=[D_f] \otimes _{C(B)}[D_b]=\ind (D_b+D_f).}
\end{proof}

Some examples of geometric situations that this theorem is applied to is introduced in Section \ref{section:4.1}.


\section{Applications}\label{section:4}

In this section we introduce some applications of the joint spectral flow and its index theorem. 

\subsection{Witten deformation and localization}\label{section:4.1}

It is easy to obtain the joint spectral flow of a continuous family of commuting Fredholm $n$-tuples when their joint spectra intersect with zero transversally. In such cases we often reduce the problem of computing the index (which usually requires to solve some linear partial differential equations or to integrate some characteristic classes) to that of counting the number of points with multiplicity.

Most typical example is the classical Poincar\'e-Hopf theorem. 

\begin{cor}[The Poincar\'e-Hopf theorem]
Let $M$ be a $Spin ^c$ manifold and $X$ a vector field on $M$ whose null points $M^X:=\mbk{p \in M \mid X(p)=0}$ are isolated. Then
$$\chi (M)=\sum _{p \in M^X} \nu _p$$
\end{cor}

This proof is essentially the same as that of Witten~\cite{Witten1982}. Here we restrict $M$ to $Spin^c$ manifolds, but it is not an essential assumption. 

\begin{proof}
By the Hodge-Kodaira decomposition the Euler characteristic $\chi (M)$ can be computed as the index of the de Rham operator $D_{\rm dR}:=d+d^* : \Gamma (\bigwedge ^{\rm even/odd} TM) \to \Gamma (\bigwedge ^{\rm odd/even} TM)$. Now $\Cliff (TM)$ acts on $\Cliff(TM)$ in two ways, $c(v)\xi:=v \cdot \xi$ and $h(v)\xi:=\gamma(\xi)\cdot v$ (for $v \in TM$ and $\xi \in \Cliff (TM)$) where $\gamma $ is the grading operator on $\Cliff (TM) \cong \Cliff (TM)^0 \oplus \Cliff (TM)^1$. They induce the $\Cliff(TM) \hat \otimes \Cliff(TM)$-action on $\Cliff(TM)$ because $c(v)$ and $h(v)$ anticommute. Because $M$ is a $Spin^c$ manifold, it is a unique irreducible $\Cliff (TM \oplus TM)$-module $\slashed{S}_\comp (TM \oplus TM)$. By Leibniz's rule, 
\ma{D_{\rm dR}(\gamma (\xi)  \cdot X)=-\gamma (D_{\rm dR} \xi) \cdot X + (-1)^{\partial \gamma (\xi )}\gamma (\xi) \cdot D_{\rm dR}(X)}
where we use the fact that $D_{\rm dR}$ is an odd operator. It means that $D_{\rm dR}$ and $h(X)$ anticommute modulo bounded operator $(-1)^{\partial \xi+1} h(D_{\rm dR}(X))$. It shows that $D_{\rm dR}+th(X)$ is Fredholm for any $t>0$ because $(D_{\rm dR}+th(X))^2 =D_{\rm dR}^2 +t^2\ssbk{X}^2 + t[D_{\rm dR},h(X)]$ is a bounded perturbation of the Laplace operator $D_{\rm dR}^2$, which is positive and has compact resolvent. On the other hand, $h(X)=\sum \ebk{e_i,X}h(e_i)$ is a commuting $n$-tuple of Fredholm operators twisted by $TM$ (now we consider $\ebk{e_i,X}$ as Fredholm operators on the $1$-dimensional vector space $\underline{\comp}$). As a consequence of Theorem \ref{thm:jsftwisted} (and Remark \ref{rem:findim}), we have
\ma{\chi(M)&=\ind (D_{\rm dR})=\ind (D_{\rm dR}+h(X))\\
&=\jsf (\mbk{\ebk{e_i,X}})=\sum _{p \in M^X} \nu _p.}
The last equation follows from the definition of the joint spectral flow.
\end{proof}

Now we consider an infinite dimensional analogue of this approach for a localization problem of index.

Let $B$ be an $n$-dimensional closed $Spin^c$ manifold, $M_1,\ldots,M_n \to B$ fiber bundles such that each fiber $Z_1$, \ldots , $Z_n$ is an odd dimensional closed manifold and $T_VM_i$ are equipped with $Spin^c$ structures, and $E$ a complex vector bundle on $M:=M_1 \times _B \cdots \times _B M_n$. Now $TB \oplus \underline{\real}^n$ is a $2n$-dimensional vector bundle and hence there is a unique $\Cliff (TB \oplus \underline{\real}^n)$-module bundle $\slashed{S}_\comp (T_VM \oplus \underline{\real}^n)$. We denote by $\slashed{S}_\comp(T_VM_i) \cong \slashed{S}_\comp^0(T_VM_i) \oplus \slashed{S}_\comp^1(T_VM_i)$ a unique $\zahl/2$-graded $\Cliff (T_VM_i)$-module bundle, which is isomorphic to $\slashed{S}_\comp(T_VM_i \oplus \underline{\real})$. Then it is decomposed as tensor products as follows. 
\ma{\slashed{S}_\comp (T_VM \oplus \underline{\real}^n) &\cong \slashed{S}_\comp (T_VM_1 \oplus \underline{\real}) \hat \otimes \cdots \hat \otimes \slashed{S}_\comp (T_VM_n \oplus \underline{\real})\\
&\cong (\slashed{S}^0_\comp (T_VM_1) \hat{\otimes} \Cliff _1) \hat \otimes \cdots \hat \otimes (\slashed{S}^0_\comp (T_VM_n) \hat \otimes \Cliff _1)\\
&\cong \bk{\slashed{S}^0_\comp (T_VM_1) \otimes \cdots \otimes \slashed{S}^0_\comp (T_VM_n)} \hat{\otimes} \Cliff _n .}
Hereafter we denote $\slashed{\mf{S}}_{\comp ,f}(M;\underline{\real}^n):= \slashed{S}_\comp (T_VM \oplus \underline{\real }^n)$ and $\slashed{S}^0_{\comp ,f} (M;\underline{\real}^n):=\slashed{S}^0_\comp (T_VM_1) \otimes \cdots \otimes \slashed{S}^0_\comp (T_VM_n)$. The inclusions $T_VM \subset T_VM \oplus \underline{\real}^n$ and $\underline{\real}^n \subset T_VM \oplus \underline{\real}^n$ induce the actions of $\Cliff(TM)$ and $\Cliff _n$ on $\slashed{\mf{S}}_{\comp ,f} (M ; \underline{\real}^n)$. Under the above identification, a vector $v=v_1 \oplus \cdots \oplus v_n \in T_VM$ acts as $(c(v_1) \otimes 1 \otimes \cdots \otimes 1) \otimes c_1  + \cdots + (1 \otimes \cdots \otimes 1 \otimes c(v_n)) \otimes c_n$ and $\Cliff _n$ acts as $1 \otimes h$ (here we denote the left and twisted right actions of $\Cliff _n$ on $\Cliff _n$ by $c$ and $h$). Hence the fiberwise Dirac operator $\slashed{D}_f$ is decomposed as
$$\slashed{D}_f=c_1\slashed{D}_1 + \cdots +c_n\slashed{D}_n ,$$
where $\slashed{D}_i$'s are Dirac operators for the $M_i$ direction
\ma{\slashed{D}_i &: \Gamma (M, \slashed{S}_\comp (T_VM \oplus \underline{\real}^n) \xra{d} \Gamma (M, \slashed{S}_\comp (T_VM \oplus \underline{\real}^n) \otimes T^*M)\\
& \hspace{5em} \xra{p_{T_VM_i}} \Gamma (\slashed{S}_\comp (T_VM \oplus \underline{\real}^n)) \xra{c} \Gamma (M,\slashed{S}_\comp (T_VM \oplus \underline{\real}^n)).}
Similarly, the twisted spinor bundle $\slashed{\mf{S}}_{\comp ,f} ^E (M; \underline{\real}^n):=\slashed{S}_\comp  (T_VM \oplus \underline{\real}^n)\otimes E$ is isomorphic to $\slashed{S}^{0,E}_{\comp ,f}(M; \underline{\real}^n) \hat \otimes  \Cliff _n$. Moreover if $E$ is equipped with a connection $\nabla ^E$ whose curvature $R^E$ satisfies $R^E(X,Y)=0$ for any $X \in T_VM_i$ and $Y \in T_VM_j $ ($i \neq j$), then the Dirac operator twisted by $E$ is decomposed as $\slashed{D}^E=c_1 \slashed{D}^E_1 + \cdots + c_n\slashed{D}^E_n$ such that $\slashed{D}_i$ commutes with $\slashed{D}_j$. Now $(\slashed{D}^E_1 , \cdots , \slashed{D}^E_n)$ forms a smooth family of unbounded commuting Fredholm $n$-tuples and $\slashed{D}_f^E$ is the smooth family of the Dirac operators associated with it.

More generally, we obtain some examples of twisted commuting Fredholm $n$-tuples. Let $V$ be a real vector bundle whose structure group is a discrete subgroup $G$ of $GL(n,\real)$ and $B'=G(V)$ a frame bundle of $V$, $M_1',\ldots,M_n'$ fiber bundles with a $G$-action on $M':=M_1' \times \cdots \times M_n'$ that is compatible with the projection $M' \to B'$, and $E$ a $G$-equivariant vector bundle on $M'$ whose connection $\nabla$ is $G$-equivariant and satisfies the above assumption. It induces a unitary representation $U_x$ of $G$ on $L^2(M'_x , \slashed{\mf{S}}_{\comp}^E(M'_x) )$ where $M'_x:={\pi'} ^{-1}(x)$ ($\pi'$ is the projection from $M'$ to $B$). We assume that 
$$U_x(g) \slashed{D}^E_i U_x(g)^* = \sum g_{ij}\slashed{D}_j^E.$$
Then $(g=(v_1,\ldots ,v_n), (\slashed{D}^E_1(x,g) , \ldots , \slashed{D}^E_n(x,g))) \in B' \times \mc{F}_n(\mc{H})$ is $G$-invariant and hence the map $x \mapsto \slashed{D}^E_v(x)$ defines a smooth family of commuting Fredholm $n$-tuples twisted by $V$. 

There are two fundamental examples. The first is the $SL(n,\zahl)$-action on $\mathbb{T}^n=(S^1)^n$ or the product bundle $\mathbb{T}^n \times B$. The second is the $\mf{S}_n$-action on the bundle $M' \times _B \cdots \times _B M'$. 
Then the Dirac operator on a fiber bundle $M:=M'/G \to B$ is that associated with $\{ \slashed{D}^E_v (x) \}_{x \in B}$.

\begin{thm}\label{thm:geom}
Let $B$, $M$, $V$, $E$, and $\nabla$ be as above. Then 
$$\ind _0 ( \slashed{D}_M^E)=\jsf \{ \slashed{D}_v^E(x) \}.$$
\end{thm}

This theorem is a direct consequence of Theorem \ref{thm:jsf} since the Dirac operator $\slashed{D}_M^E$ has the same principal symbol as $\pi ^* \slashed{\mf{D}}_B+\slashed{D}^E_f(x)$. As the special case we can show localization of the Riemann-Roch number for prequantum data on its Bohr-Sommerfeld fiber.

\begin{cor}\label{cor:FFY}
Let $(X,\omega)$ be a symplectic manifold of dimension $2n$, $\mathbb{T}^n \to X \to B$ a Lagrangian fiber bundle, and $(L,\nabla^L,h)$ its prequantum data, that is, $(L,h)$ is a hermitian line bundle over $X$ with the connection $\nabla^L$ that is compatible with $h$ whose first Chern form $c_1 (\nabla^L)$ coincides with $-2\pi i\omega$. Then its Riemann-Roch number $RR(M,L):=\ind _0 \slashed{D}_M^{\lambda ^{1/2}\otimes L}$ (where $\lambda$ is the determinant line bundle $\det T^{(1,0)}M$) coincides with the number of fibers $\mathbb{T}_x$ that $\nabla$ is trivially flat, which are called the Bohr-Sommerfeld fibers.
\end{cor}

\begin{proof}
The structure of Lagrangian fiber bundles are studied in Section 2 of \cite{Duistermaat1980} as the follows.
\begin{itemize}
\item[Fact 1.] There is a lattice bundle $P \subset TB$, which induces a flat metric on $TB$. 
\item[Fact 2.] If $P$ is trivial, $M$ is actually a principal $\mathbb{T}^n$-bundle.
\end{itemize}
We denote the $GL(n,\zahl)$-frame bundle of $TB$ by $B'$ and by $M'$ the pull-back of $M$ by the quotient $B' \to B$. It has a canonical symplectic structure and $M' \to B'$ is also a Lagrangian fiber bundle. On account of Fact 2, $M'$ is a principal $\mathbb{T}^n$-bundle on $B'$. We identify the space of constant vector fields on a fiber $M'_x$ with the Lie algebra $\mf{t} ={\rm Lie}(\mathbb{T}^n)$. 

The free $GL(n,\zahl)$-action on $B'$ extends to that on $M'$ preserving its symplectic form and affine structure on each fiber $M'_x$. Therefore it induces an action on $\mf{t}$ as $g \cdot X_i =g_{ij}X_j$ for some fixed basis $X_1,\ldots,X_n$ of $\mf{t}$. Indeed, by considering the canonical trivialization of the tangent bundle $TB' \cong B' \times \underline{\real }^n$ that is compatible with the isomorphism $\mf{t} \cong T_xB'$ given by a fixed almost complex structure $J$, we obtain the isomorphism $\mf{t} \cong T_x B' \cong \real ^n$ that is independent of the choice of $x \in B'$. Under this identification, $g \cdot : \real ^n \cong T_xB' \to T_{g \cdot x}B' \cong \real ^n$ is represented by $(g_{ij})$ as a matrix. Hence $g$ also acts on $\mf{t}$ as $(g_{ij})$.

Next, we construct some flat connections. The isomorphism $T_VM \cong T_HM \cong \pi^*TB$ induced by $J$ implies the isomorphism $\slashed{\mf{S}}_{\comp,f}(M;\underline{\real}^n) \cong \slashed{S}_\comp(M) \cong \pi ^*\slashed{\mf{S}}_\comp (B)$. Moreover it induces a flat metric on $TM$ that is trivially flat on each fiber $\mbb{T}^n$, and so are associated bundles with $TM$, in particular $\lambda ^{1/2}$ and $\slashed{S}_\comp^{\lambda ^{1/2}}(M)$. Since $R^L=c_1(L)$ is equal to $0$ when it is restricted on each fiber, $\nabla ^L$ is also fiberwise flat and the product connection $\nabla =\nabla ^{\slashed{S}_\comp^{\lambda ^{1/2} \otimes L}(M)}$ is trivially flat if and only if $\nabla ^L$ is trivially flat. 

Finally we see that $B$, $M$, $V=TB$, $E=\lambda ^{1/2} \otimes L$, and $\nabla ^{\lambda ^{1/2} \otimes L}$ satisfy the assumptions of Theorem \ref{thm:geom}. Hence $\{\nabla _v (x) \}$ forms a family of commuting Fredholm $n$-tuples twisted by $TB$ and the index of the Dirac operator $\slashed{D}_M^L$ coincides with its joint spectral flow. 

The kernel of $\Delta _f:=\nabla _{e_1}^2 + \cdots + \nabla _{e_n}^2$ is not zero if and only if $\nabla$ is, and hence $\nabla ^L$ is, trivially flat. It means that the joint spectrum of $\mbk{\nabla (x)}$ crosses over zero only on its Bohr-Sommerfeld fibers. The remaining part is that the multiplicity of eigenvalues crossing zero on each Bohr-Sommerfeld fiber is equal to $1$. It follows from the fact in symplectic geometry, that the tubular neighborhood of a Lagrangian submanifold is isomorphic to its tangent bundle as symplectic manifolds, and that $T^*\mathbb{T}^n$ is actually the product space $(T^*S^1) ^n$. More detail is in Section 6.4 of \cite{FujitaFurutaYoshida2010}.
\end{proof}

\subsection{Generalized Toeplitz index theorem}
In this section we introduce a generalization of a classical theorem relating the index of Toeplitz operators with the winding numbers.

\begin{defn}
Let $Y$ be an $n=2m-1$-dimensional closed manifold. For $\varphi : Y \to U(k)$ the generalized Toeplitz operator $T_\varphi$ is defined by
$$Pm_\varphi P : PL^2(Y,\slashed{S}_\comp(Y))^{\oplus k} \ral P L^2(Y,\slashed{S}_\comp (Y))^{\oplus k}$$
where $P$ is the orthogonal projection onto $\overline{\rm span}\mbk{\varphi  \mid \slashed{D}\varphi =\lambda \varphi \text{ for some $\lambda \geq 0$}}$.
\end{defn}

\begin{exmp}[$Y=S^1$]
In the case of $Y=S^1=\real /2\pi \zahl$ (and hence $\slashed{S}_\comp(Y)$ associated with the canonical $Spin^c$-structure on it is a trivial bundle), we can identify its Dirac operator as $d /d t$. Hence its spectrum coincides with $\zahl$ and eigenspaces $E_n$ are $1$-dimensional complex vector spaces $\comp \cdot e^{int}$. Therefore $PH=\overline{\rm span} \mbk{e^{int} ; n \in \zahl _{\geq 0}}$ and the corresponding generalized Toeplitz operators $T_\varphi$ are nothing but the ordinary ones. Its index is obtained from the winding number as $\ind T_\varphi =-{\rm winding \ } \varphi$. 
\end{exmp}

Now we generalize this index theorem for generalized Toeplitz operators in a special case.
Let $\Delta_n=\Delta _n^0 \oplus \Delta _n^1$ be a unique irreducible $\zahl /2$-graded $\Cliff _n$-module and $\gamma$ the grading operator on it. When we have a continuous map $\varphi =(\varphi _0,\ldots, \varphi _n): Y \to S^n$, we obtain an even unitary $\varphi _0 + \gamma c_1 \varphi _1 + \cdots + \gamma c_n \varphi _n$ where $c_i$ ($i=1,\ldots,n$) are Clifford multiplications of an orthonormal basis $e_1,\ldots,e_n$. For simplicity of notation, we use the same letter $\varphi$ for it restricted on $\Delta_n^0$. 

\begin{thm}\label{thm:Toep}
Let $Y$ and $\varphi $ be as above. Then
$$\ind T_\varphi =-\deg (\varphi : Y \to S^n).$$
\end{thm}

\begin{proof}
In \cite{BaumDouglas1981} Baum and Douglas proved the cohomological formula for this index which is analogous to the Atiyah-Singer formula. As a consequence, we have the following equality. 
$$\ind T_\varphi =-\ebk{\ch (\varphi) \Td (X) ,[X]}.$$
Actually we can give a proof of Theorem~\ref{thm:Toep} by using it and the description of the Chern character in Lemma \ref{lem:pair}. 
\end{proof}

\subsection{Localization of family's APS index and eta-form}\label{section:4.3}
We can also apply our joint spectral flow index theorem for fiber bundles whose fibers are compact manifolds with boundary. A main reference for this section is Melrose-Piazza~\cite{MelrosePiazza1997}.

Let $B$ be a closed $n$-dimensional manifold, $Z \to M \to B$ a smooth fiber bundle over $B$ whose boundary also forms a fiber bundle $\partial Z \to \partial M \to B$. The Riemannian metric $g$ on $TM$ is introduced by the direct sum decomposition $g_f \oplus \pi ^*g_B$ on $T_VM \oplus T_HM$, where $g_B$ is a Riemannian metric on $TB \cong T_HM$ and $g_f$ is a smooth family of Riemannian metrics on fibers $Z_x$ that are exact $b$-metrics near boundaries $\partial Z_x$. We assume that there is a $Spin ^c$-vector bundle $V$ on $B$ and $\zahl/2$-graded complex vector bundle $S$ on $M$ such that the spinor bundle $\slashed{S}_\comp(T_VM \oplus V)$ is isomorphic to $\Cliff (\pi^*V) \hat \otimes S$ as a $\Cliff(V)$-modules. Moreover the fiberwise Dirac operator $\slashed{D}_f$ on it coincides with the Dirac operator $c(v_1)D_{v_1}+\cdots +c(v_n)D_{v_n}$ associated with some $V$-twisted $n$-tuple $\mbk{D_v}$ of fiberwise first order pseudodifferential operators on $E$ that satisfies the Condition \ref{cond:commtwist}.
We denote by $H^{1,0}(M,E)$ the fiberwise Sobolev space, the completion of $C^\infty (M,E)$ by the inner product $\ebk{\cdot, \cdot }_{L^2} + \langle \nabla _f^E \cdot , \nabla _f^E \cdot \rangle$ where $\nabla _f^E := p_{T_VM} \circ \nabla ^E$. Then an element in $H^{1,0}(M,E)$ is fiberwise continuous and there is the bounded operator
$$\partial :H^{1,0}(M,E) \to L^2(\partial M,E|_{\partial M}); \ \sigma \mapsto \sigma |_{\partial M}.$$
Now we fix a spectral section $P \in C(B,\mbk{\Psi _0(\partial Z_x, E|_{\partial Z_x}))}_{x \in B})$, that is, $P$ is a projection and there is a smooth function $R:B \to \real $ such that for any $x \in B$, the condition $D_f(x) \sigma =\lambda \sigma$ implies $P(x)\sigma =\sigma$ if $\lambda >R(x)$ and $P(x)\sigma =0$ if $\lambda < -R(x)$. 
Then this $P$ determines an elliptic boundary condition at each fiber, and
\ma{\slashed{D}_f & : L^2 (M,E) \to L^2 (M,E)\\
\dom \slashed{D}_f &:= \mbk{\sigma \in H^{(1,0)}(M,E) \mid P ( \partial \sigma )=0}}
is a fiberwise Fredholm self-adjoint operator. 

Hence it forms a $V$-twisted continuous family of unbounded commuting Fredholm $n$-tuples $\mbk{D_v(x)}$ parametrized by $B$.

\begin{thm}\label{thm:jsfbdry}
 Then the following formula holds.
\ma{\ind _P(\slashed{D})=\jsf (\mbk{D(x)}) }
\end{thm}

The same proof as Theorem \ref{thm:jsf} and \ref{thm:jsftwisted} works for it. It is because we deal with operators directly, instead of the topology of its principal symbol. We only remark that in this situation $D_b$ and $D_f(1+D_f^2)^{-1/2}$ commute modulo bounded operator. Furthermore we obtain an analogue of Theorem \ref{thm:jsfopen}.

Now we introduce its application for a geometric problem.

Let $B$ be a $n$-dimensional closed manifold, $V \to B$ be a real vector bundle of dimension $n$, and $Y \to N \to B$ be a fiber bundle with $\dim Z=n-1$. We assume that $M$ can be embedded into $V$ as a fiber bundle orientedly. Then there is a fiber bundle $Z \to M \to B$ of manifold whose boundary is isomorphic to $Y \to N \to B$ as a fiber bundle. Now we define the eta-form \cite{BismutCheeger1989} for $N$ by

\ma{\hat{\eta} _P &= \int _0 ^\infty \hat{\eta}_P(t)dt \\
\hat {\eta}_P(t)&=\frac{1}{\sqrt{\pi}}{\rm Str} _{\Cliff _1}\bk{\frac{d\tilde{\mathbb{B}_t}}{dt}e^{-\tilde{\mathbb{B}}^2_t}}}
where $\tilde{\mathbb{B}}_t$ is deformed $\Cliff_1$-superconnection. 
This differential form is closed and used for the Atiah-Patodi-Singer index thoerem for families. 

On the other hand, the canonical metric on $V$ induces a smooth family of exact $b$-metrics on $T_VM$. Therefore, for first order differential operators $\partial /\partial v_i$ on $V_x$ (where $v_1 ,\ldots, v_n$ is a basis of $V_x$) form a $V$-twisted commuting Fredholm $n$-tuple when we fix a spectral section $P$. 

\begin{thm}
Let $Z \to M \to B$ and $V$ be as above. If $M$ is orientedly embeddable into $V$, its eta-form $\hat{\eta}_P$ is in $H^n(B;\zahl)$. Moreover in that case
$$\int _B \hat {\eta}_P =\ind _P (\slashed{D}_M)= \jsf \mbk{D(x)}$$
holds.
\end{thm}

\begin{proof}
From Theorem \ref{thm:jsfbdry} we have $j_*\mbk{D(x)}=\ch (\ind _P(\slashed{D}_f))$. Now the Atiyah-Patodi-Singer index theorem for families \cite{MelrosePiazza1997} says that $\ch (\ind _P(\slashed{D}_f))=\pi _!(\hat{A}(T_VM)) +\hat{\eta}_P$. In our case $T_VM$ is trivial and hence the first term of the above equality venishes. 
\end{proof}

In particular, in the case of $Y=S^{n-1}$, we get an obstruction for a sphere bundle to be isomorphic to a unit sphere of some vector bundle. It is related with the comparison of homotopy types of ${\rm Diff}_+(S^{n-1})$ and $SO(n)$, which is called the Smale conjecture.


\section{Decomposing Dirac operators}\label{section:5}
Now the converse problem arises. When are geometric Dirac operators ``decomposed'' as Dirac operators associated with commuting Fredholm $n$-tuples? In this section we deal with zeroth order pseudodifferential operators to obtain a complete obstruction from its index by using the theory of $C^*$-algebra extensions, which is related to the $KK^1$-theory in \cite{Kasparov1980} and the index theory. 

We start with a folklore. Let $T_\varphi $ be a Toeplitz operator associated with $\varphi \in C(S^1)^\times $. Then $T_\varphi $ is not a normal operator in general and $\Re T_\varphi$ commute with $\Im T_\varphi $ if and only if $\ind T_\varphi $ is equal to $0$. In this situation, the index of the operator $\Re T_\varphi +i\Im T_\varphi$ gives a complete obstruction of mutually commuting self-adjoint operators $A$ and $B$ such that $(A -\Re T_\varphi)$ and $(B-\Im T_\varphi)$ are compact. Our purpose in this section is to give an analogy and a generalization of it for the bounded operator associated with the Dirac operators. 

Before we consider the case of families, we deal with a single Dirac operator. First of all, we assume that its principal symbol is decomposed. It is interpreted as a geometric condition as follows. Let $M$ be a closed $Spin ^c$ manifold and $H_1,\ldots,H_n$ mutually orthogonal odd dimensional subbundles of $TM$ such that their direct sum spans $TM$. 
As is argued in Section \ref{section:4.1}, $\slashed{\mf{S}}_\comp (M ; \underline {\real}^n) := \slashed{S}_\comp (TM \oplus \underline{\real}^n)$ is decomposed as
\ma{\slashed{\mf{S}}_\comp (M ; \underline{\real}^n) \cong \bk{\slashed{S}^0_\comp (H_1) \otimes \cdots \otimes \slashed{S}^0_\comp (H_n)} \hat \otimes \Cliff _n.}
Hereafter we denote $\slashed{S}^0_\comp (M;\underline{\real} ^n):=\slashed{S}^0_\comp (H_1) \otimes \cdots \otimes \slashed{S}^0_\comp (H_n)$. Under this identification the principal symbol of the Dirac-type operator $\slashed{D}^E$ on $\slashed{S}_\comp^E(M;\underline{\real }^n)$ is interpreted as
\ma{\sigma (\slashed{D}^E)= \sum _{i=1}^k \bk{ \sum _{j=1}^{\dim H_i} 1 \otimes \cdots \otimes c(e_{i,j}) \xi _{i,j}\otimes \cdots \otimes 1  } \hat{\otimes}  c_i }
where each $\mbk{e_{i,j}}_{j=1,\ldots,\dim H_i}$ is an orthonormal basis on $H_i$ and $\xi _{i,j}:=\ebk{\xi , e_{i,j}}$ are coordinate functions on each cotangent space.
Then we can construct a commuting $n$-tuple in the symbol level. It also works for the Dirac operator $\slashed{D}^E$ twisted by a complex vector bundle $E$.
We say the Dirac operator $\slashed{D}^E$ is said to be {\it $n$-decomposable} if there is a bounded commuting Fredholm $n$-tuple $(T_1,\ldots,T_n)$ such that each $T_i$ is a zeroth order pseudodifferential operator on $\Gamma (M, \slashed{S}^{E,0}_\comp (M;\underline{\real} ^n))$ whose principal symbols are of the form $\sigma (T_i)=\sum _j1 \otimes \cdots \otimes c(e_{i,j}) \xi _{i,j}\otimes \cdots \otimes 1$. 
In that case the bounded operator $\slashed{D}^E(1+\slashed{D}^2)^{-1/2}$ associated with $\slashed{D}^E$ coincides modulo compact with the Dirac operator associated with the bounded commuting Fredholm $n$-tuple $T$.

In fact, $n$-decomposability is a $K$-theoretic property and determined by its index. 

\begin{prp}\label{prp:onept}
Let $M$, $H_1,\ldots,H_n$, and $E$ be as above. Then the Dirac operator $\slashed{D}^E$ is $n$-decomposable if and only if $\ind (\slashed{D}^E)=0$.
\end{prp}
\begin{proof}

A decomposition of the principal symbol gives a $*$-homomorphism $\sigma (\slashed{D}^E): C(S^{n-1}) \to A:=\Gamma (S(TM),\End  (\pi ^* \slashed{S}_\comp ^{E,0} (M ; \underline{\real} ^n)))$ that maps the coordinate function $x_i$ ($i=1,\ldots,n$) of $\real ^n$, which contains $S^{n-1}$ as the unit sphere, to an element $\sum _j c(e_{i,j}) \xi_{i,j}$. It is well-defined because the square sum $\sum _i ( \sum _j c(e_{i,j}) \xi_{i,j} )^2$ is equal to $1$. Hence we can replace the problem of obtaining a decomposition of $\slashed{D}^E$ with that of obtaining a lift, as is shown in the following diagram by the dotted arrow, of $\sigma (\slashed{D}^E)$. 
\[
\xymatrix@C=1em{
&&&C(S^{n-1}) \ar[d]^\varphi \ar@{.>}[dl]& \\
0 \ar[r] & \Psi ^{-1}(\slashed{S}_\comp ^{E,0}(M;\underline{\real} ^n)) \ar@{=}[d] \ar[r] & \Psi ^0 (\slashed{S}_\comp^{E,0}(M ;\underline{\real} ^n)) \ar[r] \ar[d] & A \ar[r]  \ar[d]^\tau& 0 \\
0 \ar[r] & \mbb{K}(\mc{H}) \ar[r] & \mbb{B}(\mc{H}) \ar[r] & Q(\mc{H}) \ar[r] & 0 
}
\]
where $\mc{H} :=L^2 (M,\slashed{S}_\comp ^{E,0} (M, \underline{\real} ^n))$ and $\Psi ^0(\slashed{S}_\comp ^{E,0}(M;\underline{\real} ^n))$ (resp. $\Psi ^{-1}(\slashed{S}_\comp ^{E,0}(M;\underline{\real} ^n))$) is the norm closure of the space of pseudodifferential operators of order $0$ (resp. $-1$). In terms of extension theory, it means that the extension $\varphi ^*\tau = \tau \circ \varphi$ is trivial. Now, as mentioned above, the theory of $C^*$-algebra extension is translated into $KK^1$-theory. In particular, a semisplit extension $\varphi$ has a lift after stabilizing by the trivial extension if and only if the $KK^1$-class $[\varphi]$ is zero. Moreover in our case we do not have to care for the stabilization of $\varphi$ because the Voiculescu theorem \cite{Voiculescu1976} ensures that $\varphi$ absorbs any trivial extension. 

In the case that $n$ is odd, it is immediately $0$ because $KK^1(C(S^{n-1}),\mbb{K})$ itself is $0$. On the other hand, $\ind \slashed{D}$ is also $0$ because $\dim M$ is odd. 

In the case that $n$ is even, we obtain an integer $\varphi ^* [\tau]  \in KK^1(C(S^{n-1}),\mbb{K}) \cong \zahl$ as the Fredholm index of $\tau \circ \varphi (u) \in Q(\mc{H})$ by Theorem 18.10.2 of \cite{Blackadar1998}. Here $u$ is the canonical generator of $KK^1(\comp ,C(S^{n-1})) \cong K_1(C(S^{n-1}))$ and its additive inverse is represented by a family of unitary matrices $u:= \sum c_1c_i x_i \in  C(S^{n-1},\End (\Delta _n^0))$ (it is a consequence of Theorem \ref{thm:Toep}). Now $\cdot \tau \circ \varphi (u)$ coincides with the principal symbol of the Dirac operator $c_1 \cdot (\slashed{D}^E)^0$ on $\Gamma (M, \slashed{S}^{E,0}(M))$ because $\slashed{S}^0_\comp(M) \cong \slashed{S}_\comp^0(M;\underline{\real}^n) \hat \otimes \Delta _n^0$.
\end{proof}

We now turn to the case of the family of Dirac operators, which is of our main interest.

Let $Z \to M \to B$ be a fiber bundle and set $n:=\dim B$. We assume that there are $Spin^c$ vector bundles $V_1,\ldots,V_l$ on $B$ and $H_1,\ldots,H_l$ on $M$ such that $\pi^*V_i \otimes H_i$ are also $Spin^c$ and the vertical tangent bundle $T_VM$ is isomorphic to their direct sum $\pi ^*V_1 \otimes H_1 \oplus \cdots \oplus \pi ^*V_l \otimes H_l$. We denote the direct sum $V_1 \oplus \cdots \oplus V_l$ by $V$ and assume $\dim V=n$. Moreover we assume that each $H_i$ is odd dimensional and decomposed as $H_i \cong H_i ^0 \oplus \underline{\real}$. Now, as is in Section \ref{section:4.1}, the spinor bundle $\slashed{\mf{S}}_{\comp ,f}(M;V):= \slashed{S}_\comp (T_VM \oplus V)$ is decomposed as
\ma{\slashed{\mf{S}}_{\comp ,f} (M ; V) \cong \bk{\slashed{S}_\comp (\pi ^*V_1 \otimes H_1^0) \otimes \cdots \otimes \slashed{S}_\comp (\pi ^*V_n \otimes H_n^0)} \hat \otimes \Cliff ( \pi ^* V).}
Hereafter we denote $\slashed{S}_{\comp ,f }^{E,0}(M;V):=\slashed{S}_\comp (\pi ^*V_1 \otimes H_1^0) \otimes \cdots \otimes \slashed{S}_\comp (\pi ^*V_n \otimes H_n^0)$. The principal symbol of the fiberwise Dirac operator $\slashed{D}_f^E$ on the twisted fiberwise spinor bundle $\slashed{\mf{S}}_{\comp ,f} ^E(M;V):=\slashed{\mf{S}}_{\comp ,f} (M;V) \otimes E$ is also decomposed as a commuting $n$-tuple twisted by $V$. Indeed, for $v=v_1 \oplus \cdots \oplus v_l$, a correspondence 
$$\sigma (\slashed{D}^E_f)_v= \sum \bk{c(v_1 \otimes e_{1,j})\xi_{ e_{1,j}}} +\cdots +\sum \bk{c(v_1 \otimes e_{l,j})\xi_{ e_{l,j}}}$$
gives the explicit decomposition. It gives a $*$-homomorphism $\sigma (\slashed{D}^E_f)_v: C(S(V)) \to C(B) \otimes Q(\mc{H})$ that is compatible with $C(B) \subset C(S(V))$ and $C(B) \otimes 1 \subset C(B) \otimes Q(\mc{H})$. In particular, when $V$ is trivial it is reduced to a $*$-homomorphism $\sigma (\slashed{D}_f^E)_v: C(S^{n-1}) \to C(B) \otimes Q(\mc{H})$.

\begin{defn}\label{def:decomp}
The fiberwise Dirac operator $\slashed{D}_f^E$ is said to be {\it $n$-decomposable} if there is a bounded commuting Fredholm $n$-tuple $\mbk{T_v (x)}$ twisted by $V$ such that each $T_v$ is a zeroth order pseudodifferential operator on $\Gamma (\slashed{S}_{\comp ,f}^{E,0} (M;V))$ whose principal symbol is $\sigma (T_v)=\sum \bk{c(v_1 \otimes e_{1,j})\xi_{ e_{1,j}}} +\cdots +\sum \bk{c(v_1 \otimes e_{l,j})\xi_{ e_{l,j}}}$.
\end{defn}

In that case $\slashed{D}^E_f (1+{\slashed{D}^E_f}^2)^{-1/2}$ coincides modulo compact operators with the smooth family of Dirac operators associated with the bounded commuting Fredholm $n$-tuples $\mbk{T_v(x)}$ twisted by $V$. Hence the $K$-class $[\ind \slashed{D}_f^E]$ is in the image of the canonical natural transform from $\tilde{k}^n(B)$ to $K^n(B)$. Moreover, the index of the Dirac operator $\slashed{D}_M^E$ on $M$ twisted by $E$, which coincides with that of $\pi ^* \slashed{\mf{D}}_B + \slashed{D}^E_f$, can be obtained from the joint spectral flow $\jsf \mbk{T_v(x)}$.

\begin{thm}\label{thm:decomp}
Let $Z \to M \to B$, $V_1,\ldots,V_l$, $H_1,\ldots ,H_l$, and $E$ be as above. Then $\slashed{D}_f^E$ is $n$-decomposable if and only if $\ind (\slashed{D}^E_f)$ is in the image of $K^n(B, B^{(n-1)}) \to K^n(B)$, or equivalently the image of $\tilde{k}^n(B) \to K^n(B)$. In that case, the equality $\ind \slashed{D}^E_M=\jsf \{ \slashed{D}^E_f \}$ holds.
\end{thm}

Here $B^{(n-1)}$ is the $(n-1)$-skelton of a cellular decomposition of $B$. The image of $K(B,B^{(n-1)}) \to K^n(B)$, which is the Atiyah-Hirzebruch filtered $K$-group $F^{n-1}K^n(B)$, is independent of the choice of decompositions and coincides with the image of $\tilde{k}^n(B) \to K^n(B)$ because of the functoriality of $\tilde{k}^* \to K^*$ and the fact that $\tilde{k}^n(B^{n-1})=0$.

\begin{remk}
In the proof, except for the last part, the condition that $B$ is an $n$-dimensional closed manifold is not necessary. Actually it is sufficient to be a finite CW-complex. Moreover, if $B$ is an $n$-dimensional CW-complex, the last part also holds.
\end{remk}

The proof is divided into some steps. First, we show that $\slashed{D}^E_f$ is locally $n$-decomposable.

\begin{lem}\label{lem:triv}
Let $M=B \times Z$ and $TZ \cong H_1 \oplus \cdots \oplus H_n$. If the index of the fiberwise Dirac operator $\slashed{D}_f^E$ on $\slashed{S}_\comp ^E (M;\real ^n)$ is zero, then it is $n$-decomposable.
\end{lem}
\begin{proof}
As in Proposition \ref{prp:onept}, it suffices to find a lift of the extension $\sigma (\slashed{D}^E_f)_v: C(S^{k-1}) \to C(B) \otimes C(S(TZ)) \subset C(B) \otimes Q(\mc{H})$. It exists when the metric on fibers are constant because $\sigma (\slashed{D}_f^E)_v$ is trivial and absorbable by Kasparov's generalized Voiculescu theorem \cite{Kasparov1980}. In general case, it exists becasuse $\sigma (\slashed{D}_f^E)_v|_{M_y}=u_y(\sigma (\slashed{D}^E_f)_v|_{M_x})u_y^*$ where $u_y: \pi^*\slashed{S}_\comp ^E (M_x;\real ^n) \to \pi^*\slashed{S}_\comp ^E (M_y;\real ^n)$ is the isometry induced from the polar part of the identity map $\id : TM_x \to TM_y$.
\end{proof}

Next we introduce a gluing technique of two decompositions. We can deal with that problem cohomologically by using the notion of Cuntz's quasihomomorphism \cite{Cuntz1983}. The ``difference'' of two lifts $\varphi _0, \varphi _1 : C(S(V)) \to C(B) \otimes \mbb{B}(\mc{H})$ of $\sigma (\slashed{D}^E_v)$ gives an element of representable $KK$-group \cite{Kasparov1988}
$$[\varphi _0, \varphi _1]:=\lbk{\hat{\mc{H}} \hat \otimes C(B), \pmx{\varphi _0 & 0 \\ 0 & \varphi _1} , \pmx{0 & 1 \\ 1 & 0}} \in \mc{R}KK(B;C(S(V)) , C(B) \otimes \mbb{K}). $$
In particular, in the case that $V$ is trivial, we can reduce the representable $KK$-group $\mc{R}KK(B;C(S(V)),C(B))$ by $KK(C(S^{n-1}),C(B) \otimes \mbb{K})$. Then the split exact sequence $ 0 \to C_0(S^{n-1}\setminus \mbk{*}) \to C(S^{n-1}) \xra{p} \comp \to 0$ gives an isomorphism $KK(C(S^{n-1}),C(F)) \cong KK(C_0(S^{n-1} \setminus \mbk{*}) ,C(F)) \oplus KK(\comp , C(F))$. When both of $\varphi _0$ and $\varphi _1$ are unital, $[\varphi _0 , \varphi _1] $ corresponds to $[\varphi _0,\varphi _1]|_{C(S^{n-1}\setminus \mbk{*})} \oplus 0$ under the above identification because $p^*[\varphi _0, \varphi _1]=[1,1]=0$. 

\begin{lem}\label{lem:glue}
Let $F_0,F_1$ be closed subsets of $B$ such that $B=(F_0)^\circ \cup (F_1)^\circ$ and $F:= F_0 \cap F_1$. We assume that $M$ and $E$ are trivial on $F$ and $\sigma (\slashed{D}^E_f)$ has lifts $\varphi _0$ and $\varphi _1$ on $F_0$ and $F_1$. Then the image of $[\varphi _0 , \varphi _1] \in KK(C_0(S^{n-1}\setminus \mbk{*}),\mbb{K}) \otimes C(F)) \cong K^{n-1}(F)$ by the boundary map of the Mayer-Vietoris sequence coincides with $[\ind \slashed{D}^E_f] \in K^n(B)$.
\end{lem}

\begin{proof}
From the diagram
\[
\xymatrix@C=1em{
0 \ar[r] &C_0(\mathbb{D} ^n \setminus \mbk{0}) \ar[r] \ar[d]^\iota & C_0(\overline{\mathbb{D}^n} \setminus \mbk{0}) \ar[r] \ar[d] & C(S^{n-1}) \ar[r] \ar@{=}[d] &0 \\
0 \ar[r] &C_0(\mathbb{D} ^n ) \ar[r]  & C_0(\overline{\mathbb{D}^n}) \ar[r]  & C(S^{n-1}) \ar[r] &0 \\
0 \ar[r] &C_0(\mathbb{D} ^n ) \ar[r] \ar@{=}[u] & C_0(\overline{\mathbb{D}^n} \setminus \mbk{*}) \ar[r] \ar[u] & C_0(S^{n-1} \setminus \mbk{*}) \ar[r] \ar[u] &0, \\
}
\]

we obtain a diagram of $KK$-groups 

\[
\xymatrix@C=1em{
KK^1(C_0(\mathbb{D}^n \setminus \mbk{0}),C(F)) \ar[r]_{\partial_1} ^\sim & KK^0(C(S^{n-1}),C(F))\\
KK^1(C(\mathbb{D}^n ),C(F)) \ar[u]^{\iota^*} \ar@{=}[d] \ar[r]_{\partial_2} & KK^0(C(S^{n-1}),C(F)) \ar@{=}[u] \ar[d] \\
KK^1(C(\mathbb{D}^n),C(F)) \ar[r]_{\partial_3 \ \ \ } ^{\sim \ \ \ }  & KK^0(C_0(S^{n-1} \setminus \mbk{*}) ,C(F)).\\
}
\]

Here for a $C^*$-algebra $A$, the group $KK^1(A ,C(F))$ is canonically isomorphic to $KK(A,\Sigma C(F)) \cong KK(A , C_0(\Sigma F))$.  One can see that boundary maps $\partial _1$ coincide with taking products with $[\id _\Sigma ] \in KK(\Sigma , \Sigma)$.

As a consequence we obtain 
$$\iota ^* \partial _3^{-1}[\varphi _0,\varphi _1]=\partial _1 ^{-1}[\varphi _0,\varphi _1] =[\varphi _0 \otimes \id _\Sigma, \varphi _1 \otimes \id _\Sigma].$$

Next we consider the isomorphism between $KK(C_0(\mathbb{D}^n) , C_0(\Sigma F))$ and $KK(\comp , C_0(\Sigma F) \hat \otimes \Cliff _n)$. As is in Section \ref{section:2}, this correspondence is given by taking a product with the canonical generator
$$[C_{\mathbb{D} ^n}]:=\lbk{C_0(\mathbb{D} ^n) \hat \otimes \Cliff _n , 1 , C_{\real ^n}:=\sum x_i \cdot c_i}$$
of $KK(\comp , C_0(\mathbb{D} ^n) \hat \otimes \Cliff _n)$. Restricting on $C_0(\mathbb{D}^n \setminus \mbk{0}) \hat \otimes \Cliff _n$, the operator $C_{\mathbb{D}^n}$ also defines an element $[C_{\mathbb{D}^n \setminus \mbk{0}}]$ in $KK(\comp , C(\mathbb{D}^n \setminus \mbk{0}) \hat \otimes \Cliff _n)$. When we regard the topological space $\mathbb{D}^n \setminus \mbk{0}$ as $\Sigma S^{n-1}$, the operator $C_{\mathbb{D}^n}$ is of the form $tC_{S^{n-1}}$ where $C_{S^{n-1}}:=\sum c_i \cdot x_i \in C(S^{n-1}) \hat \otimes \Cliff _n$ and $t$ is the identity function on $(0,1)$. Now the diagram 
\[
\xymatrix@C=1.0em@R=1.0em{
KK(C_0(\mathbb{D}^n),C_0(\Sigma F)) \ar[dd]^{\iota ^*} \ar[rd]^{[C_{\mathbb{D}^n}]} &\\
&KK(\comp , C_0(\Sigma F) \hat \otimes \Cliff _n)\\
KK(C_0(\mathbb{D}^n \setminus \mbk{0}) , C_0(\Sigma F)) \ar[ru]^{[C_{\mathbb{D}^n \setminus \mbk{0}}]}& \\
}
\]
commutes. As a consequence, we can compute $[C_{\mathbb{D}^n}] \otimes _{C_0(\mathbb{D}^n)} \partial _3^{-1}[\varphi _0 , \varphi _1]$ by using Proposition 18.10.1 of \cite{Blackadar1998} as follows.

\ma{ &[C_{\mathbb{D}^n}] \otimes _{C_0(\mathbb{D} ^n)} \partial _3 ^{-1}[\varphi _0,\varphi _1]\\
&=[C_{\mathbb{D}^n \setminus \mbk{0}}] \otimes _{C_0(\mathbb{D}^n \setminus \mbk{0})} \iota ^* \partial _3^{-1}[\varphi _0,\varphi _1]\\
&=[tC_{S^{n-1}}] \otimes _{C_0(\Sigma S ^{n-1})} [\varphi _0 \otimes \id_{\Sigma },\varphi _1 \otimes \id _{\Sigma }]\\
&=\lbk{\hat{\mc{H}} _{C_0(\Sigma F)}\hat \otimes \Cliff _n , 1 , \pmx{\varphi _0 (tC_{S^{n-1}}) & 0 \\ 0 & \varphi _1 (tC_{S^{n-1}})} \right. \\
& \left. \ \ \ \ \ \ \ \ \ \ \ + \pmx{1-\varphi _0 (tC_{S^{n-1}})^2 & 0 \\ 0 & 1-\varphi _1 (tC_{S^{n-1}})^2}\pmx{0 & 1 \\ 1 & 0}}\\
&=\lbk{\hat{\mc{H}} _{C_0(\Sigma F)}\hat \otimes \Cliff _n , 1 , \pmx{\varphi _0 (tC_{S^{n-1}}) & 1-\varphi _0 (tC_{S^{n-1}})^2 \\ 1-\varphi _1 (tC_{S^{n-1}})^2 & \varphi _1 (tC_{S^{n-1}})}}\\
&=\lbk{\hat{\mc{H}}_{C_0(\Sigma F)}\hat \otimes \Cliff _n , 1 , T}.
}
Here 
\[
T=\mbk{T_t}_{t \in [0,1]}:=
\begin{cases}
\pmx{\varphi _0 ((1-2t)C_{S^{n-1}}) & 1-(1-2t)^2 \\ 1-(1-2t)^2 & \varphi _0 ((1-2t)C_{S^{n-1}})} & \text{($0 \leq t \leq 1/2$),}\\
\pmx{\varphi _0 ((2t-1)C_{S^{n-1}}) & 1-\tilde  (2t-1)^2 \\ 1-(2t-1)^2 & \varphi _1 ((2t-1)C_{S^{n-1}})} & \text{($1/2 \leq t \leq 1$).}
\end{cases}
\]

Now we claim that this $KK$-class coincides with that comes from the cycle
$$\lbk{\mc{H}_{C_0(\Sigma F)}\hat \otimes \Cliff _n , 1 , t\varphi_0 (C_{S^{n-1}})+(1-t)\varphi _1 (C_{S^{n-1}})}.$$
Indeed, because $T_t - T_{1-t}$ is compact for any $t \in [0,1/2]$, the homotopy of continuous families of Fredholm operators
\[
\mf{T}_{s,t}:=
\begin{cases}
T_t & \text{($0 \leq t \leq s/2$),} \\
\frac{t-s/2}{1-s}T_{s/2}+ \frac{1-t-s/2}{1-s}T_{1-s/2} & \text{($s/2 \leq t \leq 1-s/2$),}\\
T_t & \text{($1-s/2 \leq t \leq 1$)}\\
\end{cases}
\]
connects $\mf{T}_0=T$ with
$$\mf{T}_1=\pmx{\varphi _0 (C_{S^{n-1}}) & 0 \\ 0 & t \varphi _0 (C_{S^{n-1}})+ (1-t) \varphi _1(C_{S^{n-1}})}.$$

Finally we obtain that $[\varphi _0 ,\varphi _1]$ coincides with $[t\varphi _0(C_{S^{n-1}})+(1-t)\varphi _1(C_{S^{n-1}})]$ in $KK(C_0(S^{n-1} \setminus \mbk{*}) , C(F)) \cong K^n(\Sigma F)$. Next we send it by the boundary map $\delta _{MV}$ of the Mayer-Vietoris exact sequence.

We denote by $I(F_0,F_1;F)$ the space $F_0 \sqcup IF \sqcup F_1$. The image of $\delta _{MV}$ is induced from the map $I(F_0,F_1;F) \to (I(F_0,F_1;F),F_0 \cup F_1)$ and excision. Therefore $\delta _{MV}[t\varphi _0(C_{S^{n-1}})+(1-t)\varphi _1(C_{S^{n-1}})]$ is of the form
\[
\begin{cases}
\varphi _0 (C_{S^{n-1}})_x & \text{($x \in F_0$)}\\
t\varphi _0 (C_{S^{n-1}})_x + (1-t)\varphi _1 (C_{S^{n-1}})_x & \text{($(x,t) \in IF$)}\\
\varphi _1 (C_{S^{n-1}})_x & \text{($x \in F_1$)} .\\
\end{cases}
\]
It is a lift of the pull-back of the principal symbol $\sigma (\slashed{D}^E_v)$ by the canonical projection $I(F_0,F_1;F) \to B$, which introduces the homotopy equivalence. As a consequence the above operator coincides with $\slashed{D}_f^E(1+(\slashed{D}^E_f)^2)^{-1/2}$ modulo compact operators and hence defines the same $KK$-class.
\end{proof}

\begin{lem}\label{lem:zero}
If $[\ind \slashed{D}^E_f] =0 \in K^n(B)$, then $\slashed{D}^E_f$ is $n$-decomposable.
\end{lem}
\begin{proof}
Let $U_1,\ldots,U_m$ be a local trivialization of the fiber bundle $M \to B$ and the vector bundles $V_1,\ldots ,V_l \to B$ such that $M$ is also trivial on $F_i:=\overline{U_i}$. By assumption and Lemma \ref{lem:triv}, $\slashed{D}^E_f$ is $n$-decomposable on each $F_i$. 

We start with the case that $B=F_0 \cup F_1$ and set $F:=F_0 \cap F_1$. First, we fix a trivial and absorbable extension $\pi : C(S^{n-1}) \to Q(\mc{H}_\pi)$ of $\mbb{K}$ by $C(S^{n-1})$ and denote by $\pi _A$ an extension $C(S^{n-1}) \to Q(\mc{H}_\pi) \to Q(\mc{H}_\pi) \otimes A$ of $C(S^{n-1})$ by $A \otimes \mbb{K}$ for a unital $C^*$-algebra $A$. 

Now we choose lifts $\varphi _0$ and $\varphi _1$ of $\sigma (\slashed{D}^E_v)$ on $F_0$ and $F_1$. By Kasparov's generalized Voiculescu theorem, $\varphi _i$'s are approximately equivalent to $\varphi \oplus \pi_{C(F_i)}$. More precisely, there are continuous families of unitaries $u_i : \mc{L}_f^2 (\slashed{S}_\comp ^E(M;V)) \to  \mc{L}_f^2 (\slashed{S}_\comp ^E(M;V)) \oplus \mc{H}_\pi \otimes C(B)$ such that $u_i (\varphi _i \oplus \pi_{C(F_i)} )u_i^* \equiv \varphi _i $ modulo compact operators. 
According to Lemma \ref{lem:glue}, $\delta_{MV} ([\varphi _0,\varphi _1])=[D_f]=0$. Hence, by exactness of the Mayer-Vietoris sequence, we have quasihomomorphisms $[\alpha _i , \beta _i]$ ($i=0,1$) such that $[\alpha _0,\beta _0]|_F-[\alpha _1,\beta _1]|_F=[\varphi _0,\varphi _1]$. Now there are unitaries $v_i$ such that $v_i (\pi_{C(F_i)} \oplus \alpha _i \oplus \alpha _i ^\perp) v_i^*\equiv \pi _{C(F_i)}$ modulo compact operators. We set 
$$\psi _i:= u_i(\varphi \oplus v_i (\pi _{C(F_i)} \oplus \beta _i \oplus \alpha _i^\perp) v_i^*)u_i^* .$$
Then $[\varphi _i , \psi _i]$ are quasihomomorphisms and $[\varphi _i,\psi _i]=[\alpha _i,\beta _i]$ in $KK(C(S^{n-1}),C(F_i) )$. 

Now $[\varphi _0,\psi _0]|_F - [\varphi _1 ,\psi _1]|_F =[\varphi _0, \varphi _1]$, which implies $[\psi _0,\psi _1]|=0$. As a consequence, there is a homotopy of quasihomomorphisms $[\Psi_0^t,\Psi _1^t]$ ($t \in [0,1]$) from $C(S^{n-1})$ to $C(F) \otimes \mbb{B}(\mc{H})$ connecting $[\psi _0 |_F , \psi _1|_F]|$ and $[\theta , \theta]$ for some $\theta$. Here we use the fact that extensions $\psi _i|_F$ contain $\pi _{C(F)}$ and hence are absorbable. 
Finally we get a homotopy $\tilde \Psi _t:=\Psi _0^{2t} (0 \leq t \leq 1/2) \Psi _1^{2-2t} (1/2 \leq t \leq 1)$ of $*$-homomorphisms from $C(S^{n-1})$ to $C(F) \otimes \mbb{B}(\mc{H})$ connecting $\psi _0$ and $\psi _1$.

Now we denote by $D$ the fiber product of $C^*$-algebras 
\[
\xymatrix@C=3.5em@R=1.5em{
D \ar[r] \ar[d] \ar@{}[rd]|\square & C(F) \otimes (B(\mc {H}) \oplus \mbb{B}(\mc{H}))  \ar[d] _{\id \otimes (p \oplus p)}\\
C(IF) \otimes Q(\mc{H}) \ar[r]^{{\rm ev}_0 \oplus {\rm ev}_1 \hspace{1.5em}} & C(F) \otimes (Q(\mc{H}) \oplus Q(\mc{H}))
}
\]
and $\tau$ the the extension
$$0 \to C_0(SF) \otimes \mbb{K} \to C(IF) \otimes \mbb{B}(\mc{H}) \to D \to 0.$$
Then $\sigma (\slashed{D}_f^E)_v$ and $(\psi _0 \oplus \psi _1)$ determine a $*$-homomorphism $\sigma :C(S^{n-1}) \to D$. Because the $C^*$-algebra $C(S^{n-1})$ is nuclear, the Choi-Effros theorem~\cite{ChoiEffros1976} implies that the pull-back $\sigma^*\tau$ is an invertible extension and hence defines an element $[\sigma ^*\tau]$ in $KK^1(C(S^{n-1}), C_0(SF) \otimes \mbb{K})$. By construction of $\tilde \Phi$, $\sigma ^*\tau$ is homotopic to the trivial extension $\pi \circ \tilde \Psi $, which implies $[\sigma ^*\tau]=0$. Consequently, $\sigma $ itself has a lift $C(S^{n-1}) \to IC(F) \otimes \mbb{B}(\mc{H})$.

Finally, we obtain a lift $\varphi$ of $\sigma (\slashed{D}_f^E)_v$ on $I(F_0,F_1;F)$. Its pull-back by a continuous section $B \to I(F_0,F_1;F)$ given by a partition of unity is a desired lift of $\sigma (\slashed{D}^E_f)_v$.

In general case we apply induction on the number of covers. We assume that there is a trivialization $B=F_1  \cup \cdots \cup F_n \cup F_{n+1}$ and set $G_0:=F_1 \cup \cdots \cup F_n$, $G_1:=F_{n+1}$. By assumption of induction, we obtain lifts $\varphi _0$ and $\varphi _1$ on $G_0$ and $G_1$. First we may assume that $V$ is trivial by restricting $\varphi _0$ the closure of an open neighborhood of $G:=G_0 \cap G_1 \subset G_0$. Now each $\varphi _i$ contains $\pi _{C(G_i)}$ by its construction. Moreover, because $M$ and $V$ are trivial on $IG$ by assumption, we can take a lift of $\sigma$ containing $\pi _{C(IG)}$. Now, the precise assertion obtained from the above argument is that if (1) $M$ and $V$ are trivial on $G$, (2) there are lifts $\varphi _i$ on $C(G_i)$ ($i=0,1$), and (3) each $\varphi_i$ is absorbable (hence it contains $\pi _{C(G_i)}$), then there is a lift $\varphi$ on $G$ containing $\pi _{C(B)}$. Hence the induction process works.
\end{proof}
Finally we prove our main theorem. Here we mention that in the above argument we restrict the case that the lift can be taken as invertible operators.
\begin{proof}[Proof of Theorem \ref{thm:decomp}]
We assume that $[\ind \slashed{D}^E_f]$ is in the image of $K^n(B,B^{(n-1)})$. Let $U \subset V$ be an inclusion of small open balls in $B$, $F_0:= U^c$, and $F_1:=\overline{V}$. Then $[\ind \slashed{D}^E_f|_{F_0}]$ and $[\ind \slashed{D}^E_f|_{F_1}]$ are $0$ by assumption and hence according to Lemma \ref{lem:zero} $\slashed{D}^E_f$ is $n$-decomposable on $F_0$ and $F_1$. Now because $F:=F_0 \cap F_1$ is homotopic to $S^{n-1}$, a group $KK(C_0(S^{n-1} \setminus \mbk{*}),C(F))$ is isomorphic to $\tilde{k}^{n-1}(F)=[C_0(\real ^{n-1}),C(F)]$. It implies that there is a $*$-homomorphism $\psi :C_0(S^{n-1} \setminus \mbk{*}) \to C(F) \otimes \mbb{K}$ such that $[\varphi _0,\varphi _1]=\Phi[\psi]$. 
Since $\varphi _1$ is absorbable, there is a unitary $u$ from $\mc{H}_{C(F)}$ to $\mc{H}_{C(F)} \oplus \mc{H}_{C(F)}$ such that $u (\varphi _1 \oplus {\rm ev}_* \cdot 1)u^* \equiv \varphi_1$ modulo compact operators. Moreover, by an argument similar to Lemma \ref{lem:zero}, we obtain a lift of $\sigma (\slashed{D}^E_f)$ on $IF$ that coincides with $\varphi _0$ on $F \times \mbk{0}$ and $u (\varphi _1 \oplus \tilde{\psi}) u^*$ on $F \times \mbk{1}$ where $\tilde \psi$ is a unital extension of $\psi$. 

The remaining part is to construct a homotopy connecting $\varphi \oplus ev_* \cdot 1$ with $\varphi \oplus \tilde \psi$. This is not realized as a family of $*$-homomorphisms on $C(S^{n-1})$ but bounded commuting Fredholm $n$-tuples. Let $\iota ^*$ be the canonical $*$-homomorphism $C(\overline{D}^n \setminus \mbk{*}) \to C_0(S^{n-1}\setminus \mbk{*})$. Then we can take a homotopy connecting $\psi \circ \iota ^*$ and $0$ since $\mathbb{D}^n$ is contractible. 

Finally, in the same way as in the proof of Lemma \ref{lem:zero}, we obtain a $*$-homomorphism $T$ that makes the following diagram commute.
\[
\xymatrix@C=4em{C(\overline{\mbb{D}}(V)) \ar[r]^{T} \ar[d] & C(B) \otimes \mbb{B}(\mc{H}) \ar[d] \\
 C(S(V)) \ar[r]^{\sigma (\slashed{D}_f^E)_v} &C(B) \otimes Q(\mc{H}).}
\]
Now $\mbk{T(x)}_v:= T(x,v)$ gives a decomposition of $\slashed{D}_f^E$.
\end{proof}

As a concluding remark, we introduce a corollary of Theorem \ref{thm:decomp}.

\begin{cor}
If $\slashed{D}_f^E$ is $n$-decomposable, then $\slashed{D}_f^{E \otimes \pi ^*F}$ is also $n$-decomposable for a complex vector bundle $F$ on $B$. Moreover in that case the following equality holds.
$$\jsf \{ \slashed{D}^{E \otimes \pi ^*F}_f \}= \dim F \cdot \jsf \{ \slashed{D}^E_f \}.$$
\end{cor}
\begin{proof}
It follows from the fact that the connective $K$-group gives a multiplicative filtration in the $K$-group. 
\end{proof}


{\small
\bibliographystyle{jalpha}
\bibliography{handmade.bib}
}
\end{document}